\documentclass[reqno, 12pt]{amsart}

\usepackage{amsthm, amsmath, amsfonts, amssymb}
\usepackage{enumerate, mathabx}
\usepackage{mathrsfs, hyperref}

\usepackage[margin=1in]{geometry}
\usepackage{graphicx}

\usepackage{cleveref}

\usepackage{algorithmicx}
\usepackage{algpseudocode}

\algdef{SE}[DOWHILE]{Do}{doWhile}{\algorithmicdo}[1]{\algorithmicwhile\ #1}%

\newtheorem{thm}{Theorem}[section]
\newtheorem{lemma}[thm]{Lemma}
\newtheorem{prop}[thm]{Proposition}
\newtheorem{cor}[thm]{Corollary}

\theoremstyle{definition}
\newtheorem{ex}{Example}

\def\R{\mathbb{E}}      \def\R{\mathbb{R}}
      
\def\Z{\mathbb{Z}}      \def\N{\mathbb{N}}
\def\T{\mathbb{T}}

\def\Om{\Omega}         \def\1{\mathds{1}}

\newcommand{\digits}{\{d_1, \dots, d_k\}}
\newcommand{\ArCantor}[2][\digits]{{C^{#1}_{#2}}}
\newcommand{\PartCantor}[2][\digits]{{\mathscr C^{#1}_{#2}}}
\newcommand{\Ellipsephic}[2][\digits]{{\mathcal E^{#1}_{#2}}}
\newcommand{\Level}[2][j]{[#2 ]_{#1}}

\newcommand{\ta}[0]{\theta}

\newcommand{\ld}[0]{\lambda}
\newcommand{\vep}[0]{\varepsilon}

\newcommand{\lsm}[0]{\lesssim}

\newcommand{\wh}[1]{\widehat{#1}}
\newcommand{\wc}[1]{\widecheck{#1}}

\newcommand{\wt}[1]{\widetilde{#1}}
\newcommand{\st}[1]{\substack{#1}}

\newcommand{\nms}[1]{\| #1 \|}
\newcommand{\E}[0]{\mathcal{E}}

\DeclareMathOperator*{\argmax}{arg\,max}

\theoremstyle{remark}
\newtheorem{rem}{Remark}

\newenvironment{enumerate-text}
{\begin{enumerate}
		\addtolength{\itemsep}{5pt}
		}
	{\end{enumerate}}

\title{Decoupling for fractal subsets of the parabola}

\author[A. Chang]{Alan Chang}
\address[AC]{Department of Mathematics, Princeton University, Princeton, NJ 08544, USA}
\email{alanchang@math.princeton.edu}
\author[J. de Dios]{Jaume de Dios Pont}
\address[JDP]{Department of  Mathematics,  University  of  California  Los  Angeles,  Portola Plaza 520, Los  Angeles,
CA 90095, USA}
\email{jdedios@math.ucla.edu}
\author[R. Greenfeld]{Rachel Greenfeld}
\address[RG]{Department of  Mathematics,  University  of  California  Los  Angeles,  Portola Plaza 520, Los  Angeles,
CA 90095, USA}
\email{greenfeld@math.ucla.edu}
\author[A. Jamneshan]{Asgar Jamneshan}
\address[AJ]{Department of Mathematics, Ko\c{c} University, Rumelifeneri Yolu, 34450, Sariyer, Istanbul, Turkey}
\email{ajamneshan@ku.edu.tr}
\author[Z.K. Li]{Zane Kun Li}
\address[ZKL]{Department of Mathematics, Indiana University Bloomington, Bloomington, IN-47405, USA}
\email{zkli@iu.edu}
\author[J. Madrid]{Jos\'e Madrid}
\address[JM]{Department of  Mathematics,  University  of  California  Los  Angeles,  Portola Plaza 520, Los  Angeles,
CA 90095, USA}
\email{jmadrid@math.ucla.edu}

\begin{document}
\begin{abstract}
We consider decoupling for a fractal subset of the parabola. We reduce studying $l^{2}L^{p}$ decoupling
for a fractal subset on the parabola $\{(t, t^2) : 0 \leq t \leq 1\}$
to studying $l^{2}L^{p/3}$ decoupling for the projection of this subset to the interval $[0, 1]$.
This generalizes the decoupling theorem of Bourgain-Demeter in the case of the parabola.
Due to the sparsity and fractal like structure, this allows us to improve upon Bourgain-Demeter's decoupling theorem
for the parabola.
In the case when $p/3$ is an even integer we derive theoretical and computational
tools to explicitly compute the associated decoupling constant for this projection to $[0, 1]$.
Our ideas are inspired by the recent work on ellipsephic sets by Biggs \cite{biggs-moment, biggs-quadratic} using nested
efficient congruencing.
\end{abstract}
\maketitle

\section{Introduction}
Fix an integer $q \geq 3$, not necessarily a prime, and let $\delta(i) := 1/q^{i}$, $i \geq 0$.
Let $C_{0} := [0, 1]$. To construct level $i$, we partition $C_{i - 1}$ into
intervals of length $\delta(i)$, remove
some of them, and denote by $N(i)$ the number of unremoved intervals.
We associate $C = \bigcap_{i \geq 0} C_{i}$ with its levels $C_i$.
For an interval $I$ with $|I| = \delta(i)$, $\delta(i) > \delta(j)$, $P_{\delta(j)}(I \cap C_j)$
will denote the collection of intervals that make up $C_j$ which are contained in $I$.
We also let $P_{\delta(i)}(C_i) = P_{\delta(i)}([0, 1] \cap C_i)$ be the collection of intervals of length $\delta(i)$ that make up $C_i$ and so $N(i) = \# P_{\delta(i)}(C_i)$.

We call $C = \bigcap_{i \geq 0} C_{i}$ a \emph{generalized Cantor set} and $C_{i}$  a \emph{generalized Cantor set of level $i$}, when  the following three conditions are satisfied:
\begin{itemize}
\item  $N(i + j) = N(i)N(j)$.
\item  $C_{i} \subset C_{i - 1}$.
\item  The level $C_{i}$ is similar to level $C_{i - 1}$. More precisely, for every interval $I\in P_{\delta(i-1)}(C_{i-1})$, the set $I\cap C_i$ is a translate of $q^{-1}C_{i-1}$.
\end{itemize}

By multiplicativity of $N(\cdot)$, given an $I \in P_{\delta(i)}(C_i)$
and $i < j$, the number of intervals in $P_{\delta(j)}(C_j)$ that are contained in $I$ is $N(j - i)$.
Additionally,
\begin{align}\label{deltaN}
\delta(i)^{-\dim(C)} = N(i)
\end{align}
where $\dim(C)$ is the Hausdorff dimension of $C$.
Note that in our definition, it is possible to let $N(i) = q^{i}$ and so $C_{i}$
is the partition of $[0, 1]$ into intervals of length $1/q^{i}$.

The traditional middle-thirds Cantor set has $q = 3$ and $N(i) = 2^{i}$.
To avoid writing \emph{generalized Cantor set} repeatedly, we will just call the above constructed set $C$, a Cantor set and $C_i$, a level of Cantor set.
A simple modification of our argument also allows it to work with asymmetric Cantor sets, however in order to simplify the arguments notation-wise, we do not pursue such a goal here.

Given a level of a Cantor set $C_{i}$, for each interval $I \in P_{\delta(i)}(C_i)$, let $\ell_{I}$
denote the left endpoint of $I$ and
\begin{align*}
\Om_{I} := \{\xi \in \R^2: \ell_{I} \leq \xi_1 \leq \ell_{I} + \delta(i), |\xi_2 - (2\ell_{I} + \delta(i))(\xi_1 - \ell_{I}) - \ell_{I}^{2}| \leq \delta(i)^{2}\}.
\end{align*}
Note that $\Om_{I}$ is a $O(\delta(i)) \times O(\delta(i)^{2})$ parallelogram
that covers and is covered by a $O(\delta(i)^{2})$ neighborhood of the piece of parabola above $I$.

For an interval $I$ and $f: \R \rightarrow \R$, let $f_{I}$ be defined such that $\wh{f_{I}} = \wh{f}1_{I}$. Next for a region $\ta$ and $f: \R^2 \rightarrow \R$, let $f_{\ta}$ be defined
such that $\wh{f_{\ta}} = \wh{f}1_{\ta}$.

Finally, throughout this paper, for two nonnegative expressions $X$ and $Y$ we use the notation $X \lsm Y$ or $Y \gtrsim X$ to denote the bound $X \leq CY$ for some absolute constant $C >0$.
If there are subscripts, for example, $X \lsm_{p} Y$, then we mean that there exists a constant $C_{p}>0$ depending only
on $p$ such that $X \leq C_{p}Y$. Additionally $X \sim Y$ means that $X \lsm Y$ and $Y \lsm X$.

\subsection{Decoupling for $C_i$ on the parabola}
Fix a Cantor set $C$ and its levels $C_i$.
For $p \geq 2$, let $D_{p}(\delta(i))$ be the best constant such that
\begin{align*}
\nms{\sum_{J \in P_{\delta(i)}(C_{i})}f_{\Om_{J}}}_{L^{p}(\R^2)} \leq D_{p}(\delta(i))(\sum_{J \in P_{\delta(i)}(C_{i})}\nms{f_{\Om_J}}_{L^{p}(\R^2)}^{2})^{1/2}
\end{align*}
for all Schwartz functions $f$ which are Fourier supported in $\bigcup_{J \in P_{\delta(i)}(C_i)}\Om_{J}$.

In the case when the Cantor set $C$ is the whole interval $[0, 1]$ and $C_i$ is the partition of $[0, 1]$
into intervals of length $\delta(i)$, we see that $D_{p}(\delta(i))$ is just the regular $l^{2}L^p$ decoupling
constant for the parabola considered by Bourgain-Demeter in \cite{BD, BD-studyguide} and so we immediately have
$D_{p}(\delta(i)) \lsm_{\vep} \delta(i)^{-\vep}(1 + \delta(i)^{-(\frac{1}{2} - \frac{3}{p})})$.
Our main result is the following generalization of Bourgain-Demeter's parabola decoupling theorem.

\begin{thm}\label{main}
Fix $p \geq 2$ and a Cantor set $C$ and its levels. Let $\kappa_{p}(C)$ be the smallest number such that
\begin{align}\label{maininput}
\nms{\sum_{J \in P_{\delta(i)}(C_{i})}f_{J}}_{L^{p}(\R)} \lsm_{p, \vep, \dim(C), N(1)} N(i)^{\kappa_{p}(C) + \vep}(\sum_{J \in P_{\delta(i)}(C_{i})}\nms{f_{J}}_{L^{p}(\R)}^{2})^{1/2}
\end{align}
for all Schwartz functions $f: \R \rightarrow \R$ and all $i$. Then the $l^{2}L^{3p}$ decoupling constant for $C$ is such that for every $\vep > 0$,
$$D_{3p}(\delta(i)) \lsm_{p, \vep, \dim(C), N(1)} N(i)^{\kappa_{p}(C) + \vep}.$$
\end{thm}
This theorem is proven in Section \ref{sec_main}.
The case of $p = 2$ is just an immediate application Bourgain-Demeter's result on the parabola and \eqref{deltaN}. For $p > 2$,
due to the sparsity and fractal structure of $C$,
we can do better than directly applying Bourgain-Demeter (see the examples summarized later or alternatively
written in more detail in Section \ref{sub:Examples}).

In the case when $C$ is the whole interval, Theorem \ref{main}
gives a sharp theorem for decoupling for the parabola.
However, whether Theorem \ref{main} is sharp for arbitrary Cantor sets $C$ is an area to be explored.
Note that even if the $\lsm_{p, \vep, \dim(C), N(1)}$ can be replaced with $\lsm_{p, \vep}$ (as is the case with our examples in Section \ref{sub:Examples}),
the proof of Theorem \ref{main} adds in implicit constants that depends on $\dim(C)$ and $N(1)$.

The proof of Theorem \ref{main} is inspired from \cite{biggs-quadratic}, in particular
one can think of \cite[(1.2)]{biggs-quadratic} as an $l^{2}L^{2t}$ decoupling theorem on the line for which we then upgrade
to an $l^{2}L^{6t}$ decoupling theorem on the parabola. However, Theorem \ref{main}
is more general than \cite{biggs-quadratic} since it is valid for arbitrary Cantor sets as defined on the first page
rather than ellipsephic sets.
Additionally, similar to the relation between \cite{biggs-moment} and \cite{biggs-quadratic}, given a Cantor set $C$ and its levels, one can use ideas from \cite{GLYZK} to write
a version of Theorem \ref{main} which upgrades $l^{2}L^{p}$ decoupling on the line to $l^{2}L^{k(k + 1)p/2}$ decoupling on the moment curve $\xi \mapsto (\xi, \xi^2, \ldots, \xi^k)$.
However in this paper we only consider the case of the parabola.

Analogous to how \cite{biggs-quadratic} is related to Wooley's nested efficient congruencing \cite{nested},
the proof of Theorem \ref{main} is similar in style to the proof of decoupling for the parabola found in \cite{GLYZK, li-efficient}
though here we more closely follow Tao's exposition \cite{tao-247} based off these two papers.
For more discussion on decoupling interpretations of efficient congruencing, see \cite{GLY, GLYZK, li-efficient} which are decoupling interpretations
of the efficient congruencing papers \cite{heathbrown}, \cite{nested}, and \cite[Section 4.3]{pierce}, respectively.

Demeter in \cite{demeter-cantor} generalized decoupling for the parabola in a different way. He considered
the partition that arises from the set $\mathcal{C}_{\alpha, n} = \{0, \alpha\}+ \{0, \alpha^2\} + \cdots + \{0, \alpha^n\}$ for $0 \leq \alpha \leq 1/2$
and proved $l^{2}L^{p}$, $2 < p < 6$ decoupling estimates for the parabola decoupling question associated to this partition.
The case $\alpha = 1/2$ corresponds to the uniform partition of $[0, 1]$ into intervals of length $2^{-n}$.
More precisely, he showed that the decoupling constant is $O_{\vep}(2^{n\vep})$ uniform in $\alpha$.
The difference between Demeter's result and our work here is that he starts with the whole interval $[0, 1]$ and decouples into a self similar partition
of $[0, 1]$ built from $\mathcal{C}_{\alpha, n}$ while in our work we start with a sparse subset of $[0, 1]$ and decouple into its individual pieces.
Additionally, the intervals in his partition have varying lengths while here our intervals all have the same length.
See also \cite{HK95} for a much stronger square function estimate for a lacunary partition of $[0, 1]$,
the same comments on \cite{demeter-cantor} also apply here.

\subsection{Decoupling for $C_i$ on $[0, 1]$}
Theorem \ref{main} reduces studying $D_{3p}(\delta(i))$ to studying \eqref{maininput}. We accomplish this in Section \ref{cantor_line}
for even integer $p$ and specific Cantor sets $C$ related to ellipsephic sets.

\subsubsection{Discrete restriction and decoupling} First we define a discrete restriction for subsets $S \subset \Z^m$ and decoupling constants for $\Om \subset [0, 1]$.
For $S \subset \Z^m$, let $A_{p, m}(S)$ be the best constant such that
\begin{align*}
\nms{\sum_{\ell \in S}a(\ell)e(\ell \cdot x)}_{L^{p}([0, 1]^m)} \leq A_{p, m}(S)(\sum_{\ell \in S}|a(\ell)|^{2})^{1/2}
\end{align*}
for all $a: S \rightarrow \R_{\geq 0}$.
Next for a subset $\Om \subset [0, 1]$ partitioned into intervals $I$ of equal length, let $K_{p}(\Om)$ be the best constant such that
\begin{align*}
\nms{\sum_{I}f_{I}}_{L^{p}(\R)} \leq K_{p}(\Om)(\sum_{I}\nms{f_I}_{L^{p}(\R)}^{2})^{1/2}
\end{align*}
for all Schwartz functions $f: \R \rightarrow \R$.

Since we plan to discuss multiple different $S$ and $S$ will be related to $\Om$, we have chosen
to emphasize the dependence of $A_{p, m}(S)$ and $K_{p}(\Om)$ on $S$ and $\Om$ rather than just the scale
that comes naturally with $\Om$. This is different from what we did in the definition of $D_{p}(\delta(i))$ above with $C_i$
being associated naturally with the scale $\delta(i)$.

\subsubsection{Arithmetic Cantor sets and ellipsephic sets}
We define an \emph{arithmetic Cantor set} of base $q$ with digits $0 \leq d_1 < \ldots < d_k < q \in \N$
to be the set of fixed points of the iterated function system generated by the functions
$\left\{f_{d_j} = \left(x\mapsto q^{-1} (x+d_j)\right)\right\}_{j=1,\dots, k}$.
This is a self-similar compact subset of $[0,1]$ with Hausdorff dimension $
\frac{\log k}{\log q}$. We will denote it by $\ArCantor{q}$.

Denote by $\Level[j]{\ArCantor{q}}$ the $j-$th level of $\ArCantor{q}$, that is
$$\Level[j]{\ArCantor{q}}:=\bigcup_{(s_1, \dots, s_j) \in \digits^j} (f_{s_1}\circ \dots \circ f_{s_j})([0,1]).$$
For brevity of notation, the intervals of length $q^{-j}$ in $P_{q^{-j}}(\Level[j]{\ArCantor{q}})$ will be denoted by $\Level[j]{\PartCantor{q}}$.
In particular, observe that
$$\Level[j]{\PartCantor{q}}=\{(f_{s_1}\circ \dots \circ f_{s_j})([0,1]): (s_1, \dots, s_j) \in \digits^j\}.$$
The standard middle thirds Cantor set is the arithmetic Cantor set  $\ArCantor[\{0,2\}]{3}$. Note also that
$\ArCantor[\{0,1\}]{3}$ and $\ArCantor[\{0,2\}]{3}$ are dilated copies of each other.

There is also a close connection between arithmetic Cantor sets and ellipsephic sets defined in \cite{biggs-quadratic}.
An \emph{ellipsephic set} of base $q$ with digits $0\le d_1< \dots <d_k < q \in \mathbb N$ is the set of integers of the form $\sum_{s=0}^{j-1} a_s q^s$
(with $a_s \in \digits$) for some $j\ge 1$. We will denote it by $\Ellipsephic{q}$. We will use $\Level[j]{\Ellipsephic{q}}$ to mean the set $\Ellipsephic{q}\cap [0,q^{j})$.
Comparing the definitions of an arithmetic Cantor set and an ellipsephic set,
we easily observe that $$\Level[j]{\ArCantor{q}} =  q^{-j}\left(\Level[j]{\Ellipsephic{q}}+[0,1]\right).$$

Using the convenience that $2n$ is even and expanding the $L^{2n}$ norm (\Cref{prop:decoupling-ellipsephic-optimization}),
allows use to show \Cref{prop:ellipsephic_to_cantor}
\begin{align}\label{k2na2n}
K_{2n}(\Level[j]{\PartCantor q}) \sim A_{2n, 1}(\Level[j]{\Ellipsephic q})
\end{align}
(where the implied constant is absolute) which connects decoupling and discrete restriction constants.

When we study $A_{2n, 1}(\Level[j]{\Ellipsephic q})$, we will say $\Ellipsephic{q}$ has no carryover if $nd_k < q$. In particular,
this definition depends on the $n$ in question. Additionally note that we will say that $\Ellipsephic{q}$ has carryover if $nd_k \geq q$.
This terminology was inspired from the proof of \cite[Lemma 2.2]{biggs-quadratic}.
Using Freiman isomorphisms, we have the following nice proposition which simplifies greatly discrete restriction
for ellipsephic sets when we have no carryover (see \Cref{prop:no_carryover} for a more precise statement).
\begin{prop}\label{nocarrover_intro}
If $\Ellipsephic{q}$ is an ellipsephic set without carryover, then $$A_{2n, 1}( [\Ellipsephic{q}]_j ) = A_{2n, 1}( [\Ellipsephic{q}]_1 )^j.$$
\end{prop}

\begin{rem}
{\L}aba and Wang in \cite{labawang} consider a restriction estimate for a certain kind of fractal measure in $\R^d$.
The main ingredient in the proof of their main theorem is a decoupling estimate for a particular type of Cantor set on the line built
out of a $\Lambda(p)$-set (see Lemma 5, Section 4, and Proposition 1 of \cite{labawang} for more details, see also \cite{bourgain-lambda} for the existence of $\Lambda(p)$ sets).
The techniques by which they upgrade a $\Lambda(p)$ set to a Cantor set multiscale decoupling theorem on the line can probably also be applied in our case,
though here the point of view we take is more algebraic and is closer in spirit to the number theoretic side of things.
\end{rem}

\subsection{Examples} As an illustration of the the tools developed above we can consider the case when $n = 2$ and then very explicitly study $A_{4, 1}(\Level[j]{\Ellipsephic q})$
as \Cref{prop:decoupling-ellipsephic-optimization} turns such study into an optimization problem subject to a quadratic constraint which we can very explicitly
compute. This combined with \eqref{k2na2n} allows us to upgrade $l^2 L^{4}$ discrete restriction for an ellipsephic set to $l^{2}L^{4}$ decoupling
for an arithmetic Cantor set.
In particular, below is a summary of Examples \ref{01q}-\ref{squares} we derived in Section \ref{sub:Examples}.
\vspace{0.15in}
\begin{center}
\renewcommand{\arraystretch}{1.5}
\begin{tabular}{|c|c|c|c|}
\hline
  $C_{i}$ & $\delta(i)$ & $N(i)$ & $K_{4}(C_{i})$\\
  \hline
  $[C_{q}^{\{0, 1\}}]_{i}, q > 2$ & $q^{-i}$ & $2^{i}$ & $\sim (2^{i})^{\frac{1}{4}\log_{2}(3/2)}$\\
  $[C_{3}^{\{0, 2\}}]_{i}$ & $3^{-i}$ & $2^{i}$ & $\sim (2^{i})^{\frac{1}{4}\log_{2}(3/2)}$\\
  $[C_{q}^{\{0, 1, 2\}}]_{i}, q > 4$ & $q^{-i}$ & $3^{i}$ & $\sim (3^{i})^{\frac{1}{4}\log_{3}(15/7)}$\\
  $[C_{q}^{\{0, 1, 3\}}]_{i}, q > 6$ & $q^{-i}$ & $3^{i}$ & $\sim (3^{i})^{\frac{1}{4}\log_{3}(5/3)}$\\
  $[C_{q}^{\{0^{2}, 1^{2}, \ldots, \lfloor \sqrt{q}\rfloor^{2}\}}]_{i}, q \geq \exp(\exp(O(\frac{1}{\vep})))$ & $q^{-i}$ & $(\lfloor \sqrt{q}\rfloor + 1)^{i}$ & $\lsm_{\vep} N(i)^{\vep}$\\
  \hline
\end{tabular}
\end{center}
\vspace{0.15in}

Note that from the proof of these examples in Section \ref{sub:Examples}, the implied constants do not depend on $\dim(C)$ or $N(1)$.
We only studied the $n = 2$ case out for convenience to demonstrate our methods but it is not a serious constraint.
\begin{rem}\label{squaresremark}
The ellipsephic set associated to the Cantor set in the last row of the table above was considered by Biggs in \cite[Corollary 1.4]{biggs-quadratic}.
The result in that row should be read as follows: Fix an arbitrary $\vep > 0$. Choose an integer $q \geq \exp(\exp(O(1/\vep)))$ and consider
$[C_{q}^{\{0^{2}, 1^{2}, \ldots, \lfloor \sqrt{q}\rfloor^{2}\}}]_{i}$. Note that here the Cantor set depends on $q$ and so also $\vep$.
Then we showed that the $l^{2}L^{4}$ decoupling constant for level $i$ of this Cantor set is $\lsm_{\vep} N(i)^{\vep}$ where $N(i) = (\lfloor q\rfloor + 1)^{i}$.
\end{rem}
\begin{rem}
The example in the second row of the table above is associated to the ellipsephic set $[\E_{3}^{\{0, 2\}}]_{j}$ which does have carryover.
However, the map $x \mapsto x/2$ is a Freiman isomorphism between $[\E_{3}^{\{0, 2\}}]_{j}$ and $[\E_{3}^{\{0, 1\}}]_{j}$
and the latter ellipsephic set does not have carryover. Since Freiman isomorphisms do not change numerology (see
the equality case of \eqref{eq:freiman-hom-decoupling}), the
numerology of the second row is the same as that of the first row.
\end{rem}

\begin{rem}
Note that $C_{q}^{\{0, 1, 2\}}$ and $C_{q}^{\{0, 1, 3\}}$ for $q > 6$ have the same Hausdorff
dimension but their associated $l^{2}L^{4}$ decoupling constants are different. In \Cref{maxexp} we show that
given a Hausdorff dimension $d = \log_{s}r$ with $0 < d < 1$ and $r, s \in \N$, there exists an arithmetic Cantor set $C$
such that the associated decoupling exponent $\kappa_{2n}(C)$ as defined in \eqref{maininput} is as large as possible. This means
that for arbitrary arithmetic Cantor sets $K_{2n}(C)$ does not just depend on the Cantor set, but rather also on arithmetic
properties of the set.
\end{rem}

\begin{rem}
A careful look at the proof of Example \ref{012q} (the third row in the table above) shows curiously that the optimizer of discrete restriction
for $[\E_{q}^{\{0, 1, 2\}}]_{1}$, $q > 4$ (and hence also $[\mathcal{E}_{q}^{\{0, 1, 2\}}]_{j}$ by \Cref{prop:no_carryover} because of lack of carryover).
This is different from the other examples in Section \ref{sub:Examples} and the observation that
the choice of $a: \{1, \ldots, \N\} \rightarrow \R_{\geq 0}$ being the constant function below witnesses the case of equality of the estimates
$$\nms{\sum_{1 \leq \ell \leq N}a(\ell)e(\ell x)}_{L^{2n}([0, 1])} \leq N^{\frac{1}{2} - \frac{1}{n}}(\sum_{1 \leq \ell \leq N}|a(\ell)|^{2})^{1/2}$$
and
$$\nms{\sum_{1 \leq n \leq N}a(n)e(nx + n^{2}t)}_{L^{6}([0, 1]^2)} \lsm_{\vep} N^{\vep}(\sum_{1 \leq n \leq N}|a(n)|^{2})^{1/2}$$
for all $\{a(n)\} \in \ell^{2}(\N)$.
This example suggests potential differences between discrete restriction and solution counting problems in certain cases.
\end{rem}

In the table below we feed our results into Theorem \ref{main}. Each row should be compared to the estimate
that $D_{12}(\delta(i)) \lsm_{\vep} \delta(i)^{-1/4 - \vep}$ obtained from a direct application of Bourgain-Demeter's decoupling
theorem for the parabola.
\vspace{0.15in}
\begin{center}
\renewcommand{\arraystretch}{1.5}
\begin{tabular}{|c|c|c|c|}
\hline
  $C_{i}$ & $\delta(i)$ & $N(i)$ & Applying Theorem \ref{main}\\
  \hline
  $[C_{q}^{\{0, 1\}}]_{i}, q > 2$ & $q^{-i}$ & $2^{i}$ & $D_{12}(\delta(i)) \lsm_{\vep, \dim(C)} (2^{i})^{\frac{1}{4}\log_{2}(3/2) + \vep}$\\
  $[C_{3}^{\{0, 2\}}]_{i}$ & $3^{-i}$ & $2^{i}$ & $D_{12}(\delta(i)) \lsm_{\vep} (2^{i})^{\frac{1}{4}\log_{2}(3/2) + \vep}$\\
  $[C_{q}^{\{0, 1, 2\}}]_{i}, q > 4$ & $q^{-i}$ & $3^{i}$ & $D_{12}(\delta(i)) \lsm_{\vep, \dim(C)} (3^{i})^{\frac{1}{4}\log_{3}(15/7) + \vep}$\\
  $[C_{q}^{\{0, 1, 3\}}]_{i}, q > 6$ & $q^{-i}$ & $3^{i}$ & $D_{12}(\delta(i)) \lsm_{\vep, \dim(C)} (3^{i})^{\frac{1}{4}\log_{3}(5/3) + \vep}$\\
  $[C_{q}^{\{0^{2}, 1^{2}, \ldots, \lfloor \sqrt{q}\rfloor^{2}\}}]_{i}, q \geq \exp(\exp(O(\frac{1}{\vep})))$ & $q^{-i}$ & $(\lfloor \sqrt{q}\rfloor + 1)^{i}$ & $D_{12}(\delta(i)) \lsm_{\vep, N(1)} N(i)^{\vep}$\\
  \hline
\end{tabular}
\end{center}
\vspace{0.15in}

Note that in the first four rows we have $N(1) \sim 1$ while in the second and last row we have $\dim(C) \sim 1$.
Whether our estimates for $D_{12}(\delta(i))$ above are sharp remain an area to be explored (in other words, for example, is there an $f$ Fourier supported
in $\bigcup_{J \in [\mathscr{C}_{3}^{\{0, 2\}}]_{i}}\Om_{J}$ such that $D_{12}(\delta(i)) \gtrsim (2^{i})^{\frac{1}{4}\log_{2}(3/2)}$).
Continuing the discussion in Remark \ref{squaresremark}, the last row in the table above should be compared to  \cite[Corollary 1.4]{biggs-quadratic}.

Finally the above methods are very efficient in studying the case when the ellipsephic set does not have carryover and some cases with
carryover but which are Freiman isomorphic to a case which has no carryover.
To study the case when the ellipsehic set has carryover we develop an approximation (\Cref{prop: DecExp estimate}) which allows
us to numerically approximate the $l^{2}L^{2n}$ decoupling constant on $[0, 1]$ for a given arithmetic Cantor set (see Section \ref{computational}
for more details).

\subsection{Application to solution counting}
We end with some applications of our estimates to number theory, in particular to solution counting in Vinogradov systems.

\subsubsection{The Cantor set $C_{3}^{\{0, 1\}}$}\label{c301}
Consider $[C_{3}^{\{0, 1\}}]_{j}$ and the associated ellipsephic set $[\E_{3}^{\{0, 1\}}]_{j}$.
Note $\#[\E_{3}^{\{0, 1\}}]_{j} \sim 2^i$.
We first obtained that $K_{4}([C_{3}^{\{0, 1\}}]_{j}) \sim A_{4}([\E_{3}^{\{0, 1\}}]_{j}) \sim (3/2)^{j/4}$.
This immediately implies that the number of 4-tuples to $$x_1 + x_2 = x_3 + x_4$$
with $1 \leq x_i \leq 3^{j}$ and $x_i \in [\E_{3}^{\{0, 1\}}]_{j}$ is $(3/2)^{j}2^{2j} = 6^{j}$.
This should be compared to solving $x_1 + x_2 = x_3 + x_4$ where $1 \leq x_i \leq 2^{j}$ which would give $8^{j}$ such 4-tuples.
The 6 in $6^{j}$ can be explained by the fact that since $\E_{3}^{\{0, 1\}}$ in this
case has no carryover ($2 \cdot 1 < 3$), we can look one digit at a time and
there are 6 solutions to $a + b = c + d $ where $a, b, c, d \in \{0, 1\}$.

Next we obtained that $D_{12}(\delta(j)) \lsm_{\vep} (3/2)^{j/4 + \vep}$ where $\delta(j) = 3^{-j}$.
Using the standard reduction from decoupling estimates to solving Vinogradov \cite{BDG} we see that the
number of solutions to the system
\begin{equation}\label{j62}
\begin{aligned}
x_1 + x_2 + \cdots + x_6 &= y_1 + y_2 + \cdots + y_6\\
x_{1}^{2} + x_{2}^{2} + \cdots + x_{6}^{2} &= y_{1}^{2} + y_{2}^{2} + \cdots + y_{6}^{2}
\end{aligned}
\end{equation}
where $1 \leq x_i, y_i \leq 3^{j}$ and $x_i, y_i \in [\E_{3}^{\{0, 1\}}]_{j}$ is $\lsm_{\vep} (\frac{3}{2})^{3j + \vep}2^{6j} = 6^{3j + O(\vep)}$.
This should be compared to the lower bound of $O(2^{6j})$ coming from the diagonal solutions.

\subsubsection{The Cantor set $C_{q}^{\{0^{2}, 1^{2}, \ldots, \lfloor \sqrt{q}\rfloor^{2}\}}$}\label{csquare}
Fix arbitrary $\vep > 0$. Choose $q$ an integer (not necessarily prime) such that $q \geq \exp(\exp(O(1/\vep)))$
and consider the ellipsephic set $[\E_{q}^{\{0^{2}, 1^{2}, \ldots, \lfloor \sqrt{q}\rfloor^{2}\}}]_{j}$ associated to the Cantor set
$[C_{q}^{\{0^{2}, 1^{2}, \ldots, \lfloor \sqrt{q}\rfloor^{2}\}}]_{j}$. Then the estimate that
$D_{12}(\delta(j)) \lsm_{\vep, N(1)}N(j)^{\vep} $ implies that the number of solutions to the system \eqref{j62}
where $1 \leq x_i, y_i \leq q^{j}$ and $x_{i}, y_{i} \in [\E_{q}^{\{0^{2}, 1^{2}, \ldots, \lfloor \sqrt{q}\rfloor^{2}\}}]_{j}$ is
$\lsm_{\vep, N(1)} N(j)^{6 + \vep} $. This rederives the implication obtained in \cite[Corollary 1.4]{biggs-quadratic} (where our $N(j)$ is her $Y$).

\begin{rem}
In the system considered in Section \ref{c301}, our upper bound is quite large compared to the lower bound of $2^{6N}$ which come from
the diagonal contribution. In the following, we argue that given an ellipsephic set (whose associated Cantor set has
dimension $d$), then when the number of variables is sufficiently large depending on $d$, then the contribution
of the non-diagonal solutions will be greater than that of the diagonal solutions.

More precisely, fix an arbitrary arithmetic Cantor set $C_{q}^{\{d_1, \ldots, d_k\}}$ with Hausdorff dimension $d \in (0, 1)$
and consider the associated ellipsephic set $\E_{X} := [\E_{q}^{\{d_1, \ldots, d_k\}}]_{j}$ where we have
written $X = q^{j}$.
Then $\# \E_{X} \sim X^{d}$.
We consider the question of how many solutions are there to the system
\begin{equation}\label{vino}
\begin{aligned}
x_{1} + x_2 + \cdots + x_{s} &= x_{s + 1} + x_{s + 2} + \cdots + x_{2s}\\
x_{1}^{2} + x_{2}^{2} + \cdots + x_{s}^{2} &= x_{s + 1}^{2} + x_{s + 2}^{2} + \cdots + x_{2s}^{2}
\end{aligned}
\end{equation}
where $x_{i} \in \E_{X}$. The contribution from the diagonal solutions is $O(X^{sd})$.
We claim that for sufficiently large $s$ there will always be more than $O(X^{sd})$ many solutions.

Consider the map
\begin{align*}
\Sigma : (\E_{X})^{s} &\longrightarrow [-sX, sX] \times [-sX^{2}, sX^{2}]\\
(a_1, a_2, \ldots, a_s) &\longmapsto (a_1 + \cdots + a_s, a_{1}^{2} + \cdots + a_{s}^{2})
\end{align*}
The map $\Sigma$ goes from a set of cardinality $O(X^{sd})$ to a set of cardinality $O(s^{2}X^{3})$. For notational convenience let $A_X = [-sX, sX] \times [-sX^{2}, sX^{2}]$. The number of solutions $J_s(X)$ to  \eqref{vino} is bounded below by:

\begin{align*}
    J_s(X) =&
    \sum_{(n_1,n_2)\in  A_X} (\sum_{\substack{ a_1^j+\dots+a_s^j = n_j\\ a_i \in (\E_{X})^{s}, j=1,2}} 1)^2
    \\\ge&
    |A_X|^{-1}
(\sum_{(n_1,n_2)\in  A_X} \sum_{\substack{ a_1^j+\dots+a_s^j = n_j\\ a_i \in (\E_{X})^{s}, j=1,2}} 1)^2
\\ = & (O(s^2X^{3}))^{-1} \cdot (O(X^{sd}))^2 = O(X^{2sd-3}/s^2)
\end{align*}

Therefore the number of solutions to \eqref{vino} is at least $O(X^{2sd-3}/s^2)$. Comparing this to the number of diagonal solutions $O(X^{sd})$ shows that
for $s$ sufficiently large (depending on Hausdorff dimension), the contribution
of the off-diagonal solutions are more than the diagonal solutions.
\end{rem}

\subsection*{Acknowledgements}
JD was partially supported by “La Caixa” Fellowship LCF/ BQ/ AA17/ 11610013.
RG was partially supported by the Eric and Wendy Schmidt Postdoctoral Award.
AJ was supported by DFG-research fellowship JA 2512/3-1.
ZL is supported by NSF grant DMS-1902763. ZL is also grateful to the Department
of Mathematics at the University of Chicago and the University of California, Los Angeles
for their hospitality when he visited in February 2020.

The authors would also like to thank Iqra Altaf, Kirsti Biggs, Julia Brandes, Ciprian Demeter, Bingyang Hu, and Terence Tao for helpful comments, discussions, and suggestions.

\section{Proof of Theorem \ref{main}}\label{sec_main}
Fix a Cantor set $C$ (and its levels).
Much like the proof of decoupling for the parabola in \cite{li-efficient}, the proof of \Cref{main} reduces to four lemmas: parabolic rescaling, bilinear reduction, the key estimate,
and H\"{o}lder's inequality.

\subsection{Parabolic rescaling and bilinear reduction}
We first start with the parabolic rescaling lemma. The proof is fairly standard, but we include it here for convenience.
\begin{lemma}[Parabolic rescaling]\label{parab}
Suppose $0 \leq \delta(j) \leq \delta(i) \leq 1$ and $I \in P_{\delta(i)}(C_i)$. Then
\begin{align}\label{parabeq1}
\nms{\sum_{J \in P_{\delta(j)}(I \cap C_j)}f_{\Om_{J}}}_{L^{p}(\R^2)} \leq D_{p}(\delta(j - i))
\left(\sum_{J \in P_{\delta(j)}(I \cap C_j)}\nms{f_{\Om_J}}_{L^{p}(\R^2)}^{2}\right)^{1/2}.
\end{align}
\end{lemma}
\begin{proof}
Write $I = [a, a + \delta(i)]$. Consider the ``Galilean transform'' $S_I : \R^2 \to \R^2$  represented by the matrix
\begin{align*}
\begin{pmatrix}
\delta(i)^{-1} & 0\\ 0 & \delta(i)^{-2}
\end{pmatrix}
\begin{pmatrix}
1 & 0\\
-2a & 1
\end{pmatrix}
.
\end{align*}
The key geometric observation is that since $C_{i}$ is a level of a Cantor set (and Cantor set levels are similar), we have a bijection $P_{\delta(j)}(I \cap C_j) \to P_{\delta(j - i)}(C_{j - i})$ given by $J \mapsto J' = \delta(i)^{-1}(J-a)$, and furthermore,
\begin{align}
\label{eq:S_I-Om-J}
S_{I}(\Om_{J} - (a, a^2)) = \Om_{J'}.
\end{align}
Define $g_{I}(y) := f(S_{I}^{\top}y)e(-S_{I}(a, a^2) \cdot y)$, so that $\wh{g_{I}}(\eta) = \delta(i)^{3} \wh{f}(S_{I}^{-1}\eta + (a, a^2))$. With $J, J'$ as above, we have
\begin{align*}
f_{\Om_{J}}(x)
&=
\int_{\Om_{J}}\wh{f}(\xi)e(\xi \cdot x)\, d\xi\\
&=
e(x \cdot (a, a^2))\int_{\Om_{J'}}\wh{g_{I}}(\eta)e(\eta\cdot (S_{I}^{-1})^\top x)\, d\eta
=
e(x \cdot (a, a^2))(g_{I})_{\Om_{J'}}((S_{I}^{-1})^{\top}x)
\end{align*}
where in the second equality we made the change of variables $\eta = S_{I}(\xi - (a, a^2))$ and used \eqref{eq:S_I-Om-J}. Therefore,
\begin{align*}
|\sum_{J \in P_{\delta(j)}(I \cap C_j)}f_{\Om_{J}}(x)| = |\sum_{J' \in P_{\delta(j - i)}(C_{j - i})}(g_{I})_{\Om_{J'}}((S_{I}^{-1})^{\top}x)|
\end{align*}
and hence
\begin{align*}
\nms{\sum_{J \in P_{\delta(j)}(I \cap C_j)}f_{\Om_{J}}}_{L^{p}(\R^2)} &= \delta(i)^{-3/p}\nms{\sum_{J' \in P_{\delta(j - i)}(C_{j - i})}(g_{I})_{\Om_{J'}}}_{L^{p}(\R^2)}\\
&\leq \delta(i)^{-3/p}D_{p}(\delta(j - i))(\sum_{J' \in P_{\delta(j - i)}(C_{j - i})}\nms{(g_{I})_{\Om_{J'}}}_{L^{p}(\R^2)}^{2})^{1/2}.
\end{align*}
Reversing all the change of variables then obtains the right hand side of \eqref{parabeq1}.
\end{proof}

Parabolic rescaling implies the following immediate corollary.
\begin{cor}[Almost multiplicativity]\label{almostmult}
We have $$D_{p}(\delta(i + j)) \leq D_{p}(\delta(i))D_{p}(\delta(j)).$$
\end{cor}

Next we define the following bilinear constant.
Let $0 \leq \delta(j) \leq \delta(i_1), \delta(i_2) \leq \delta(k) \leq 1$. Let
$M_{p}(j, k, i_1, i_2)$ to be the best constant such that one has the estimate
\begin{align*}
\int_{\R^2}
&|\sum_{J_1 \in P_{\delta(j)}(I_1 \cap C_j)}f_{\Om_{J_1}}|^{p}|\sum_{J_2 \in P_{\delta(j)}(I_2 \cap C_j)}g_{\Om_{J_2}}|^{2p} \\
&\leq M_{p}(j, k, i_1, i_2)^{3p}(\sum_{J_1 \in P_{\delta(j)}(I_1 \cap C_j)}\nms{f_{\Om_{J_1}}}_{L^{3p}(\R^2)}^{2})^{p/2}(\sum_{J_2 \in P_{\delta(j)}(I_2 \cap C_j)}\nms{g_{\Om_{J_2}}}_{L^{3p}(\R^2)}^{2})^{p}
\end{align*}
for all $I_1 \in P_{\delta(i_1)}(C_{i_1})$ and $I_2 \in P_{\delta(i_2)}(C_{i_2})$ such that $d(I_1, I_2) \geq \delta(k)$
and all Schwartz functions $f$ with Fourier support on $\bigcup_{J_{1} \in P_{\delta(j)}(I_1 \cap C_j)}\Om_{J_1}$
and Schwartz functions $g$ with Fourier support on $\bigcup_{J_2 \in P_{\delta(j)}(I_2 \cap C_j)}\Om_{J_2}$.
Note that from H\"{o}lder,
\begin{align}\label{trivialmult}
M_{p}(j, k, i_1, i_2) \leq D_{3p}(\delta(j - i_1))^{1/3}D_{3p}(\delta(j - i_2))^{2/3}.
\end{align}

\begin{lemma}[Bilinear reduction]\label{bilinear}
If $0 \leq \delta(j) \leq \delta(i) \leq 1$, then
\begin{align}\label{red}
D_{3p}(\delta(j)) \lsm D_{3p}(\delta(j - i)) + N(i)^{O(1)}M_{p}(j, i, i, i).
\end{align}
\end{lemma}
\begin{proof}
Fix a Schwartz function $f$ with Fourier support in $\bigcup_{J \in P_{\delta(j)}(C_j)}\Om_{J}$.
We have
\begin{align}
\notag
\|\sum_{J \in P_{\delta(j)}(C_j)}f_{\Om_J}\|_{L^{3p}(\R^2)}^2
&
=
\nms{\sum_{\st{I_1, I_2 \in P_{\delta(i)}(C_i)}}
\left(\sum_{J_1 \in P_{\delta(j)}(I_1 \cap C_j)}f_{\Om_{J_1}}\sum_{J_2 \in P_{\delta(j)}(I_2 \cap C_j)}f_{\Om_{J_2}}\right)}_{L^{3p/2}(\R^2)}
\\
&
\label{redeq1}
\leq
\nms{
\sum_{\st{I_1, I_2 \in P_{\delta(i)}(C_i)\\d(I_1, I_2) \leq \delta(i)}} ( \cdots )
}_{L^{3p/2}(\R^2)}
+
\nms{
\sum_{\st{I_1, I_2 \in P_{\delta(i)}(C_i)\\d(I_1, I_2) \geq \delta(i)}} ( \cdots )
}_{L^{3p/2}(\R^2)}
\end{align}
By multiple applications of the Cauchy-Schwarz inequality, the first term of \eqref{redeq1} is
\begin{align*}
&\leq \sum_{\st{I_1, I_2 \in P_{\delta(i)}(C_i)\\d(I_1, I_2) \leq \delta(i)}}\nms{\sum_{J_1 \in P_{\delta(j)}(I_1 \cap C_j)}f_{\Om_{J_1}}}_{L^{3p}(\R^2)}\nms{\sum_{J_2 \in P_{\delta(j)}(I_2 \cap C_j)}f_{\Om_{J_2}}}_{L^{3p}(\R^2)}\\
&\leq \big(\sum_{I_1 \in P_{\delta(i)}(C_i)}\nms{\sum_{J_1 \in P_{\delta(j)}(I_1 \cap C_j)}f_{\Om_{J_1}}}_{L^{3p}(\R^2)}^{2}\big)^{1/2}\times \\
&\hspace{1.5in}\big(\sum_{I_1 \in P_{\delta(i)}(C_i)}\big(\sum_{\st{I_2 \in P_{\delta(i)}(C_i)\\d(I_1, I_2) \leq \delta(i)}}\nms{\sum_{J_2 \in P_{\delta(j)}(I_2 \cap C_j)}f_{\Om_{J_2}}}_{L^{3p}(\R^2)})^{2}\big)^{1/2}\\
&\lsm \sum_{I \in P_{\delta(i)}(C_i)}\nms{\sum_{J \in P_{\delta(j)}(I \cap C_j)}f_{\Om_{J}}}_{L^{3p}(\R^2)}^{2}\\
&\leq D_{3p}(\delta(j - i))^2\sum_{J \in P_{\delta(j)}(C_j)}\nms{f_{\Om_J}}_{L^{3p}(\R^2)}^{2}.
\end{align*}
In the third inequality above, we used the fact that for a fixed $I_1$, the number of $I_2$ satisfying $d(I_1, I_2) \leq \delta(i)$ is $\lesssim 1$. In the last inequality above, we  applied the definition of $D_{3p}(\delta(j - i))$.
This gives the first term on the right hand side of \eqref{red}.
The second term of \eqref{redeq1} is
\begin{align}\label{second}
\lsm N(i)^{O(1)}\max_{\st{I_1, I_2 \in P_{\delta(i)}(C_i)\\d(I_1, I_2) \geq \delta(i)}}\nms{(\sum_{J_1 \in P_{\delta(j)}(I_1 \cap C_j)}f_{\Om_{J_1}})(\sum_{J_2 \in P_{\delta(j)}(I_2 \cap C_j)}f_{\Om_{J_2}})}_{L^{3p/2}(\R^2)}.
\end{align}
For any two nonnegative functions $F,G$, we have $\int F^{3p/2}G^{3p/2} \leq (\int F^{p}G^{2p})^{1/2}(\int F^{2p}G^{p})^{1/2}$ by Cauchy-Schwarz. Using this observation and applying the definition of $M_{p}(j, i, i, i)^{3p}$
gives that \eqref{second} is
\begin{align*}
&\leq
N(i)^{O(1)}M_{p}(j, i, i, i)^2\times\\
&\hspace{0.5in}\max_{\st{I_1, I_2 \in P_{\delta(i)}(C_i)\\d(I_1, I_2) \geq \delta(i)}}(\sum_{J_1 \in P_{\delta(i)}(I_1 \cap C_j)}\nms{f_{\Om_{J_1}}}_{L^{3p}(\R^2)}^{2})^{1/2}(\sum_{J_2 \in P_{\delta(i)}(I_2 \cap C_j)}\nms{f_{\Om_{J_2}}}_{L^{3p}(\R^2)}^{2})^{1/2}\\
&\leq N(i)^{O(1)}M_{p}(j, i, i, i)^2(\sum_{J \in P_{\delta(j)}(C_j)}\nms{f_{\Om_J}}_{L^{3p}(\R^2)}^{2}).
\end{align*}
This gives the second term of the right hand side of \eqref{red} and thus completes the proof of the lemma.
\end{proof}

\subsection{Key Estimate}
The main idea of this section is that while the key estimate for the proof of decoupling for the parabola in \cite{li-efficient}
follows from Plancherel (see \cite[Lemma 3.8]{GLYZK} with $k = 2$, \cite[Remark 4]{li-efficient}, or \cite[Proposition 19]{tao-247}), the key estimate here will follow from \eqref{maininput}.
\begin{lemma}[Key estimate]\label{key}
If $0 \leq \delta(j) \leq \delta(i_1), \delta(i_{1}'), \delta(i_2) \leq \delta(k) \leq 1$
with $\delta(i_2)^{2} \leq \delta(i_{1}') \leq \delta(i_1)$, then for any $\vep > 0$,
\begin{align*}
M_{p}(j, k, i_1, i_2) \lsm_{p, \vep, \dim(C), N(1)} \delta(k)^{-O(1)}M_{p}(j, k, i_{1}', i_2)N(i_{1}' - i_{1})^{\kappa_{p}(C)/3 + \vep/3}
\end{align*}
where $\kappa_{p}(C)$ is defined in \eqref{maininput}.
\end{lemma}
\begin{proof}
Fix arbitrary $\vep > 0$ and arbitrary $I_1 \in P_{\delta(i_1)}(C_{i_1})$ and $I_2 \in P_{\delta(i_2)}(C_{i_2})$
such that $d(I_1, I_2) \geq \delta(k)$.
Next fix arbitrary Schwartz functions $f$ and $g$ with Fourier support in $\bigcup_{J_1 \in P_{\delta(j)}(I_1 \cap C_j)}\Om_{J_1}$
and $\bigcup_{J_2 \in P_{\delta(j)}(I_2 \cap C_j)}\Om_{J_2}$, respectively. We may normalize
$f$ and $g$ so that
\begin{align}\label{fgnorm}
\sum_{J_1 \in P_{\delta(j)}(I_1 \cap C_j)}\nms{f_{\Om_{J_1}}}_{L^{3p}(\R^2)}^{2} = \sum_{J_2 \in P_{\delta(j)}(I_2 \cap C_j)}\nms{g_{\Om_{J_2}}}_{L^{3p}(\R^2)}^{2} = 1.
\end{align}
Thus we need to show that
\begin{equation*}
\begin{aligned}
\int_{\R^2}|\sum_{J_1 \in P_{\delta(j)}(I_1 \cap C_j)}f_{\Om_{J_1}}|^{p}&|\sum_{J_2 \in P_{\delta(j)}(I_2 \cap C_j)}g_{\Om_{J_2}}|^{2p}\\
& \lsm_{p, \vep, \dim(C), N(1)} \delta(k)^{-O(p)}N(i_{1}' - i_{1})^{p\kappa_{p}(C) + p\vep}M_{p}(j, k, i_{1}', i_2)^{3p}.
\end{aligned}
\end{equation*}
Write $I_1 := [a, a + \delta(i_1)]$ and $I_2 := [b, b + \delta(i_2)]$.
Assume that $I_2$ is to the left of $I_1$ and so $a - b > \delta(k)$; the case when $I_2$ is to the right of $I_1$ is similar.

We now essentially reduce to the case when $b = 0$.
To see this, let $T_{I_2} = (\begin{smallmatrix} 1 & 0\\-2b & 1 \end{smallmatrix})$,
$\wt{f}_{I_2}(y) := f(T_{I_2}^{\top}y)e(-y \cdot T_{I_2}(b, b^2))$, and $\wt{g}_{I_2}(y) := g(T_{I_2}^{\top}y)e(-y \cdot T_{I_2}(b, b^2))$.
By a similar argument as in the proof of \Cref{parab}, it suffices to show that
\begin{equation}\label{keyeq0a}
\begin{aligned}
\int_{\R^2}|\sum_{J_1 \in P_{\delta(j)}((I_1 - b) \cap (C_{j} - b))}(\wt{f}_{I_2})_{\Om_{J_1}}|^{p}|&\sum_{J_2 \in P_{\delta(j)}([0, \delta(i_2)] \cap (C_{j} - b))}(\wt{g}_{I_2})_{\Om_{J_2}}|^{2p}\\
& \lsm_{p, \vep, \dim(C), N(1)} \delta(k)^{-O(p)}N(i_{1}' - i_{1})^{p\kappa_{p}(C) + p\vep}M_{p}(j, k, i_{1}', i_2)^{3p}
\end{aligned}
\end{equation}
where
\begin{align*}
\sum_{J_1 \in P_{\delta(j)}((I_1 - b) \cap (C_{j} - b))}\nms{(\wt{f}_{I_2})_{\Om_{J_1}}}_{L^{3p}(\R^2)}^{2} = \sum_{J_2 \in P_{\delta(j)}([0, \delta(i_2)] \cap (C_{j} - b))}\nms{(\wt{g}_{I_2})_{\Om_{J_2}}}_{L^{3p}(\R^2)}^{2} = 1
\end{align*}
since $\det T_{I_2} = 1$.

\begin{figure}
    \centering
    \includegraphics[width=.8\textwidth]{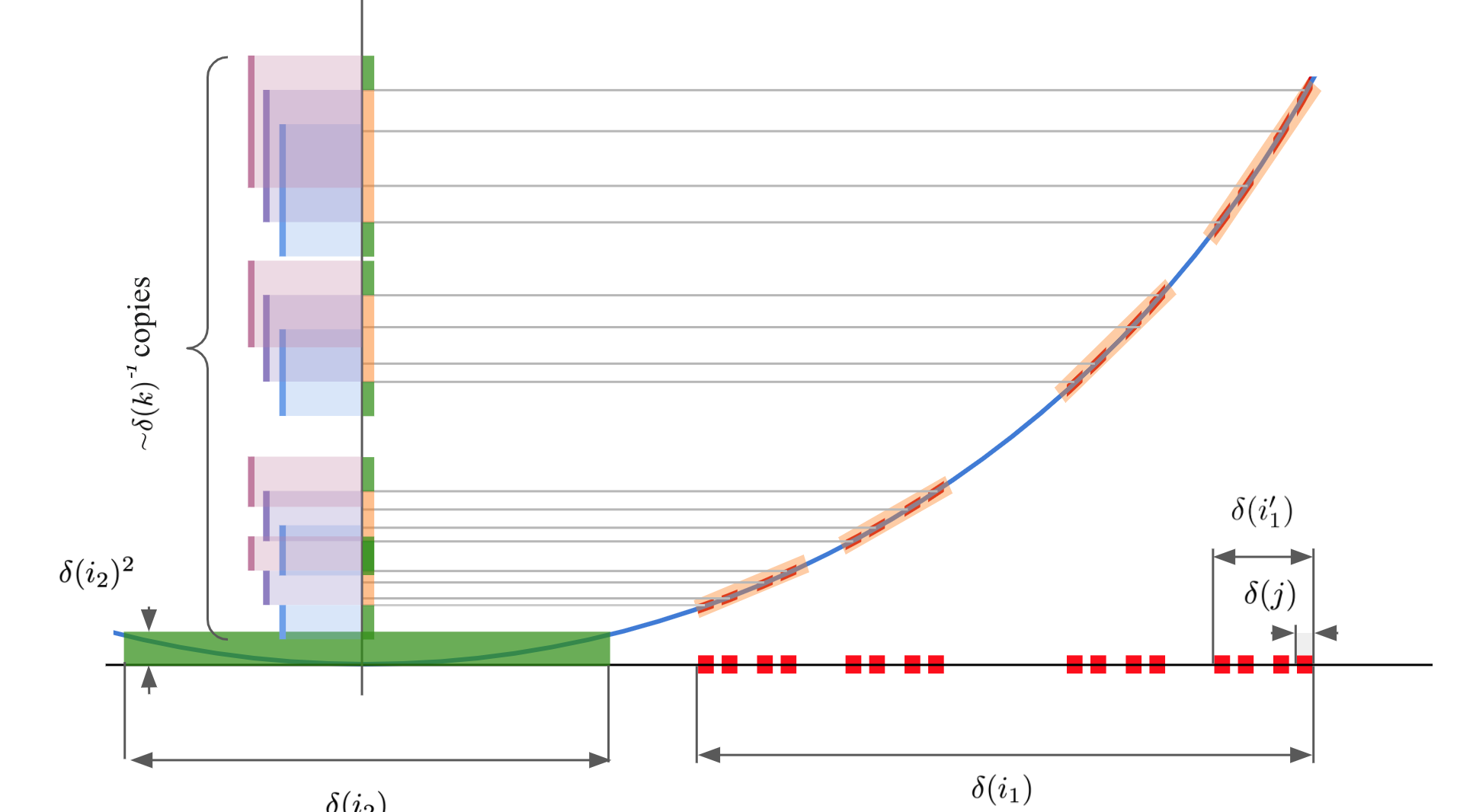}
    \caption{Scheme of the key estimate. Since $I_1$ is away from the origin and the parabola is Lipschitz on $I_1$ with Lipschitz constant $\gtrsim \delta(k)^{-O(1)}$,
    we know we can decouple vertically.
    The fact that we are multiplying by $G^2$, on the Fourier side amounts to convolving against $\wh{G} \ast \wh{G}$,
    which adds an uncertainty of size $O(\delta(i_2)^{2})$ on each vertical level. This is acceptable because,
    we can cover the overlap by $\delta(k)^{-1}$ many copies of the orange sets (these copies are in shades of blue, purple
    and maroon in the picture).}
    \label{fig:curved_cantor}
\end{figure}

Let $$G := \sum_{J_2 \in P_{\delta(j)}([0, \delta(i_2)] \cap (C_{j} - b))}(\wt{g}_{I_2})_{\Om_{J_2}}.$$ Then $G$ (and hence $G^2$)
is Fourier supported in an $O(\delta(i_2)) \times O(\delta(i_2)^{2} + \delta(j))$ rectangle centered at the origin.
For each $J \in P_{\delta(i_{1}')}((I_1 - b) \cap (C_{i_{1}'} - b))$, let $$F_{J} := \sum_{J_1 \in P_{\delta(j)}(J \cap (C_{j} - b))}(\wt{f}_{I_2})_{\Om_{J_1}}.$$
The Fourier transform of $F_J$ is supported in the horizontal strip $\{(\xi_1, \xi_2): \xi_{2} = \gamma_{J}^{2} + O(\delta(i_{1}'))\}$ where $\gamma_{J}$ is the center of $J$
and $\gamma_{J}$ is a distance $\gtrsim \delta(k)$ away from the origin.
Since $\delta(j), \delta(i_{2})^{2} \leq \delta(i_{1}')$,
$F_{J}G^{2}$ has Fourier transform supported in the horizontal strip $\{(\xi_1, \xi_2): \xi_{2} = \gamma_{J}^{2} + O(\delta(i_{1}'))\}$ as well.

Using this notation, showing \eqref{keyeq0a} is equivalent  to showing that
\begin{align}
\begin{aligned}\label{reduction}
\int_{\R^2}|\sum_{J \in P_{\delta(i_{1}')}((I_1 - b) \cap (C_{i_{1}'} - b))}F_{J}G^{2}|^{p} \lsm_{p, \vep, \dim(C), N(1)} \delta(k)^{-O(p)}N(i_{1}' - i_{1})^{p\kappa_{p}(C) + p\vep}M_{p}(j, k, i_{1}', i_2)^{3p}.
\end{aligned}
\end{align}
We now claim that
\begin{equation}
\begin{aligned}\label{lowerdim}
\|&\sum_{J \in P_{\delta(i_{1}')}((I_1 - b) \cap (C_{i_{1}'} - b))}F_{J}G^{2}\|_{L^{p}(\R^2)}\\
&\hspace{0.5in}\lsm_{p, \vep, \dim(C), N(1)} \delta(k)^{-O(1)}N(i_{1}' - i_{1})^{\kappa_{p}(C) + \vep}(\sum_{J \in P_{\delta(i_{1}')}((I_1 - b) \cap (C_{i_{1}'} - b))}\nms{F_{J}G^{2}}_{L^{p}(\R^2)}^{2})^{1/2}
\end{aligned}
\end{equation}
which, as we will show, follows from an application of Cantor set decoupling for the line given by \eqref{maininput}.
Let us see how to use \eqref{lowerdim} to prove \eqref{reduction}.
Reversing the change of variables used to obtain \eqref{keyeq0a} and applying the definition of $M_{p}(j, k, i_{1}', i_{2})$ along with the normalization of $g$ in \eqref{fgnorm} gives
\begin{align}\label{keyeq1}
\nms{F_{J}G^2}_{L^{p}(\R^2)} \leq M_{p}(j, k, i_{1}', i_2)^{3}(\sum_{J_1 \in P_{\delta(j)}((J + b) \cap C_j)}\nms{f_{\Om_{J_1}}}_{L^{3p}(\R^2)}^{2})^{1/2}
\end{align}
for each $J \in P_{\delta(i_{1}')}((I_1 - b) \cap (C_{i_{1}'} - b))$.
Combining \eqref{lowerdim} with \eqref{keyeq1} and using our normalization of $f$ in \eqref{fgnorm} then proves \eqref{reduction}.
Thus it remains to prove \eqref{lowerdim}.

First since $p \geq 2$, by Minkowski's inequality, it suffices to prove that for fixed $x \in \R^2$,
\begin{equation}\label{keyeq2}
\begin{aligned}
&\int_{\R}|\sum_{J \in P_{\delta(i_{1}')}((I_1 - b) \cap (C_{i_{1}'} - b))}F_{J}(x, y)G(x, y)^{2}|^{p}\, dy\\
& \lsm_{p, \vep, \dim(C), N(1)} \delta(k)^{-O(p)}N(i_{1}' - i_{1})^{p\kappa_{p}(C) + p\vep}(\sum_{J \in P_{\delta(i_{1}')}((I_1 - b) \cap (C_{i_{1}'} - b))}(\int_{\R}|F_{J}(x, y)G(x, y)^{2}|^{p}\, dy)^{2/p})^{p/2}.
\end{aligned}
\end{equation}
Indeed, once we obtain the above inequality, we can prove \eqref{lowerdim} by just integrating in $x$.
For fixed $x$, the Fourier transform in $y$ of $F_{J}(x, y)G(x, y)^{2}$ is supported on an interval of length
$O(\delta(i_{1}'))$ centered at $\gamma_{J}^{2}$ where $\gamma_J \gtrsim \delta(k)$ is the center of the interval $J \in P_{\delta(i_{1}')}((I_1 - b) \cap (C_{i_{1}'} - b))$.
Note that the implied constant in $O(\delta(i_{1}'))$ is independent of $J$.

Now suppose $F_{J_1}G^2$ and $F_{J_2}G^2$ had overlapping Fourier supports. Then $\gamma_{J_1}^{2} = \gamma_{J_2}^{2} + O(\delta(i_{1}'))$
and hence $\gamma_{J_1} = \gamma_{J_2} + O(\delta(i_{1}')\delta(k)^{-O(1)})$ since $\gamma_{J_1}, \gamma_{J_2} \gtrsim \delta(k)$.
Thus \eqref{keyeq2} now follows if we can show that
\begin{align*}
\int_{\R}|&\sum_{J \in P_{\delta(i_{1}')}((I_1 - b) \cap (C_{i_{1}'} - b))}f_{c J}(y)|^{p}\, dy\\
&\lsm_{p, \vep, \dim(C), N(1)} \delta(k)^{-O(p)}N(i_{1}' - i_{1})^{p\kappa_{p}(C) + p\vep}(\sum_{J \in P_{\delta(i_{1}')}((I_1 - b) \cap (C_{i_{1}'} - b))}(\int_{\R}|f_{c J}(y)|^{p}\, dy)^{2/p})^{p/2}
\end{align*}
for $1 \leq c \lsm \delta(k)^{-O(1)}$ and for arbitrary Schwartz functions $f$. Here, $cJ$ denotes the interval having the same center as $J$ but of length $c|J|$. By rescaling $I_1$ and using the fact that decoupling constants are translation invariant, this then reduces to showing that
\begin{align}
\label{eq:key-est-last-reduction}
\nms{\sum_{J \in P_{\delta(i)}(C_{i})}f_{cJ}}_{L^{p}(\R)} \lsm_{p, \vep, \dim(C), N(1)} c N(i)^{\kappa_{p}(C) + \vep}(\sum_{J \in P_{\delta(i)}(C_{i})}\nms{f_{cJ}}_{L^{p}(\R)}^{2})^{1/2}
\end{align}
for $c \geq 1$ and for arbitrary Schwartz functions $f$. (Here $i = i_1' - i_1$.)

To show \eqref{eq:key-est-last-reduction}, we can assume that $c\geq 1$ is an integer. We can find translations $\{\tau_{k} : 1 \leq k \leq c\}$ such that
for any $J \in P_{\delta(i)}(C_{i})$, the interval $cJ$ is covered by the union of $\{\tau_{k}(J) : 1 \leq k \leq c\}$.
Therefore
\begin{align*}
\nms{\sum_{J \in P_{\delta(i)}(C_{i})}f_{cJ}}_{L^{p}(\R)}
&= \nms{\sum_{k=1}^c\sum_{J \in P_{\delta(i)}(C_{i})}(f_{cJ})_{\tau_{k}(J)}}_{L^{p}(\R)}\\
&\leq c\sup_{k}\nms{\sum_{J \in P_{\delta(i)}(C_{i})}(f_{cJ})_{\tau_{k}(J)}}_{L^{p}(\R)}\\
&\lsm_{p, \vep, \dim(C), N(1)} c N(i)^{\kappa_{p}(C) + \vep}\sup_{k}(\sum_{J \in P_{\delta(i)}(C_{i})}\nms{(f_{cJ})_{\tau_{k}(J)}}_{L^{p}(\R)}^{2})^{1/2}\\
&\lsm_{p, \vep, \dim(C), N(1)} c N(i)^{\kappa_{p}(C) + \vep}(\sum_{J \in P_{\delta(i)}(C_{i})}\nms{f_{cJ}}_{L^{p}(\R)}^{2})^{1/2}
\end{align*}
where the third inequality is because decoupling is invariant under translation and \eqref{maininput}, and the last inequality is by boundedness of the Hilbert transform in $L^{p}(\R)$, $1 < p < \infty$, (see for example \cite[p. 59]{duoandikoetxea}).
This completes the proof of \eqref{eq:key-est-last-reduction} and hence the proof of \Cref{key}.
\end{proof}

\subsection{The iteration}
We first have the following lemma which allows us to interchange the last two indices in $M_{p}(j, k, i_1, i_2)$.
\begin{lemma}\label{interchange}
If $0 \leq \delta(j) \leq \delta(i_1) \leq \delta(i_2) \leq \delta(k) \leq 1$, then
\begin{align*}
M_{p}(j, k, i_1, i_2) \leq M_{p}(j, k, i_2, i_1)^{1/2}D_{3p}(\delta(j - i_2))^{1/2}.
\end{align*}
\end{lemma}
\begin{proof}
This lemma follows from $\int F^{p}G^{2p} \leq (\int F^{2p}G^{p})^{1/2}(\int G^{3p})^{1/2}$
and applying the definition of $M_{p}(j, k, i_2, i_1)$ and parabolic rescaling.
\end{proof}

We are now in a good position to conclude the proof of Theorem \ref{main}. After normalization, the iteration is essentially the same as in \cite{li-efficient}. The proof follows via a contradiction argument, combining the previous lemmas and using an iteration argument.
We start normalizing the main objects that we have been considering in order to simplify our argument.
Let
$$D_{3p}'(\delta(i)) := N(i)^{-\kappa_{p}(C)}D_{3p}(\delta(i))$$
and
\begin{align*}
M_{p}'(j, k, i_1, i_2) := M_{p}(j, k, i_1, i_2)(N(j - i_1)N(j - i_2)^{2})^{-\kappa_{p}(C)/3}.
\end{align*}
With this definition, after multiplying both sides of Lemma \ref{bilinear} by $N(j - i)^{-\kappa_{p}(C)}$,
we have that if $0 \leq \delta(j) \leq \delta(i) \leq 1$, then
\begin{align}\label{reduction-n}
D_{3p}'(\delta(j)) \lsm N(i)^{-\kappa_{p}(C)}D_{3p}'(\delta(j - i)) + N(i)^{O(1)}M_{p}'(j, i, i, i).
\end{align}
The key estimate Lemma \ref{key} now becomes that if $0 \leq \delta(j) \leq \delta(i_1), \delta(i_{1}'), \delta(i_2) \leq \delta(k) \leq 1$
with $\delta(i_{2})^2 \leq \delta(i_{1}') \leq \delta(i_1)$, then for any $\vep > 0$,
\begin{align}\label{key-n}
M_{p}'(j, k, i_1, i_2) \lsm_{p, \vep, \dim(C), N(1)}\delta(k)^{-A}N(i_{1}' - i_1)^{\vep/3}M_{p}'(j, k, i_{1}', i_2)
\end{align}
for some absolute constant $A$.
Also, Lemma \ref{interchange} above  becomes
\begin{align}\label{interchange-n}
M_{p}'(j, k, i_1, i_2) \leq M_{p}'(j, k, i_2, i_1)^{1/2}D_{3p}'(\delta(j - i_2))^{1/2}.
\end{align}

\begin{proof}[Proof of Theorem \ref{main}]
Let $\ld$ be the least exponent for which the following statement is true:
\begin{equation}\label{lambda}
    D'_{3p}(\delta(j)) \lsm_{p, \vep, \dim(C), N(1)} N(j)^{\ld + \vep}
    \qquad\text{for all $j \geq 0$ and $\vep > 0$.}
\end{equation}
Trivially, $D_{3p}'(\delta(i)) \leq N(i)^{\frac{1}{2} - \kappa_{3p}(C)}$ and so \eqref{lambda} is equivalent
to the statement that
\begin{equation*}
D'_{3p}(\delta(j)) \lsm_{p, \vep, \dim(C), N(1)} N(j)^{\ld + \vep}
\qquad\text{for all $j \gtrsim 1$ and $0 < \vep \lsm 1$.}
\end{equation*}

If $\ld = 0$, then we are done, so we assume towards a contradiction that $\lambda > 0$.
Fix arbitrary $\vep > 0$, we may assume that $\vep < 1$.

If $1 \leq a \leq \frac{j}{4i}$, then $j \geq 4ai \geq 2ai \geq ai \geq i$ which imply that
we can talk about $M_{p}'(j, i, 2ai, i)$ and $M_{p}'(j, i, 4ai, 2ai)$. Applying \eqref{interchange-n}, \eqref{key-n}, and \eqref{lambda} in that order obtains
\begin{align*}
M'_{p}(j,i,2ai,ai)
&\leq
M'_{p}(j,i,ai,2ai)^{1/2}D'_{3p}(\delta(j-ai))^{1/2}
\\
&\lesssim_{p, \vep, \dim(C), N(1)}
M'_{p}(j,i,4ai,2ai)^{1/2}\delta(i)^{-A/2}N(4ai-ai)^{\vep/6}D'_{3p}(\delta(j-ai))^{1/2}
\\
&\lesssim_{p, \vep, \dim(C), N(1)}
M'_{p}(j,i,4ai,2ai)^{1/2}\delta(i)^{-A/2}N(4ai-ai)^{\vep/6}N(j-ai)^{\frac{\lambda}{2}+\frac{\vep}{2}}
\\
&= M'_{p}(j,i,4ai,2ai)^{1/2}\delta(i)^{-A/2}N(j)^{\frac{\ld + \vep}{2}}N(i)^{-a\ld/2}.
\end{align*}
Hence we have shown that for $1 \leq a \leq \frac{j}{4i}$ %
\begin{align*}
\label{eq:iteration-claim1}
M'_{p}(j,i,2ai,ai)\leq C_{p, \vep, \dim(C), N(1)} M'_{p}(j,i,4ai,2ai)^{1/2}\delta(i)^{-A/2}N(i)^{-a\ld/2}N(j)^{\frac{\ld + \vep}{2}}
\end{align*}
for some constant $C_{p, \vep, \dim(C), N(1)}$ depending only on $p, \vep$, $\dim(C)$ and $N(1)$ and $A$ is an absolute constant.

Then, we multiply both sides of the previous inequality by $N(j)^{-\lambda}$ and raise both sides to the $1/a$ power to obtain that for every integer $a$ such that $1 \leq a \leq \frac{j}{4i}$,
\begin{equation*}
\begin{split}
(N(j)^{-\ld}&M_{p}'(j, i, 2ai, ai))^{1/a}\\
& \leq (C_{p, \vep, \dim(C), N(1)}\delta(i)^{-A/2}N(j)^{\vep/2})^{1/a}N(i)^{-\ld/2}(N(j)^{-\ld}M_{p}'(j, i, 4ai, 2ai))^{1/(2a)}.
\end{split}
\end{equation*}
Therefore, for all $k\in\N$ with $2^{k + 1} \leq j/i$, the following inequality holds:
\begin{align}\label{iteration core}
N(&j)^{-\lambda} M'_{p}(j,i,2i,i)\nonumber\\
&\leq
\left(
\prod_{n=0}^{k-1}
(C_{p, \vep, \dim(C), N(1)} \delta(i)^{-A/2}N(j)^{\vep/2})^{1/2^n}
\right)
N(i)^{-k\lambda/2}
\left(
N(j)^{-\lambda}
M'_{p}(j, i,{2^{k+1}i}, {2^k}i)
\right)^{1/2^k}
\nonumber\\
&\lsm_{p, \vep, \dim(C), N(1)}
(\delta(i)^{-A/2} N(j)^{\vep/2})^{\sum_{n = 0}^{k - 1}\frac{1}{2^n}}
N(i)^{-k\lambda/2}
N(j)^{\vep/2^k}
\nonumber\\
&\lesssim_{p, \vep, \dim(C), N(1)} \delta(i)^{-O(1)}N(i)^{-k\ld/2}N(j)^{\vep}
\end{align}
where in the second inequality we have used that
\begin{align*}
M_{p}'(j, i, 2^{k + 1}i, 2^{k}i) &\leq D_{3p}'(\delta(j - 2^{k + 1}i))^{1/3}D_{3p}'(\delta(j - 2^{k}i))^{2/3}\\
&\lsm_{p, \vep, \dim(C), N(1)} N(j - 2^{k + 1}i)^{(\ld + \vep)/3}N(j - 2^{k}i)^{2(\ld + \vep)/3} \leq N(j)^{\ld + \vep}
\end{align*}
which follows from \eqref{trivialmult} and that $N$ is increasing.

Suppose $i$, $j$, and $k$ are such that $N(i) = N(j)^{1/2^{k + 1}}$ and so by multiplicativity of $N(\cdot)$, $2^{k + 1}i = j$.
Using \eqref{deltaN}, \eqref{reduction-n}, \eqref{key-n}, \eqref{lambda} and \eqref{iteration core} we conclude that
\begin{align*}
    D'_{3p}(\delta(j))&\lesssim_{p, \vep, \dim(C), N(1)} N(i)^{-\kappa_{p}(C)}D_{3p}'(\delta(j - i)) + \delta(i)^{-O(1)}N(i)^{\vep}M_{p}'(j, i, 2i, i)\notag
    \\
    &\lesssim_{p, \vep, \dim(C), N(1)} N(i)^{-\kappa_{p}(C)}N(j - i)^{\ld + \vep} + \delta(i)^{-O(1)}N(i)^{\vep - k\ld/2}N(j)^{\ld + \vep}\notag\\
    &\lsm_{p, \vep, \dim(C), N(1)} N(j)^{\ld + \vep}N(i)^{- \ld} + N(i)^{O(\frac{1}{\dim(C)}) + \vep - k\ld/2}N(j)^{\ld + \vep}\notag
    \\
    &\lesssim_{p, \vep, \dim(C), N(1)} N(j)^{\ld(1 - \frac{1}{2^{k + 1}}) + \vep}+N(j)^{\ld[1 - \frac{1}{2^{k + 1}}(\frac{k}{2} - \frac{O(\frac{1}{\dim(C)})}{\ld} - \frac{\vep}{\ld})]}N(j)^{\vep}.%
\end{align*}

Choose $K$ so that $\frac{K}{2} - \frac{O(\frac{1}{\dim(C)})}{\ld} - \frac{\vep}{\ld} \geq 1$. We have then shown that if $j = 2^{K + 1}\N$,
then for every $\vep > 0$, $$D'_{3p}(\delta(j)) \lsm_{p, \vep, \dim(C), N(1)} N(j)^{\ld(1 - \frac{1}{2^{K + 1}}) + \vep}.$$

We now upgrade this to be a statement for all $j \geq 0$.
We use almost multiplicativity, \Cref{almostmult}.
For $n \geq 0$ and $j$ such that $2^{K + 1}n \leq j \leq 2^{K + 1}(n + 1)$. Note that
$$N(2^{K + 1}n) \leq N(j) \leq N(2^{K + 1}(n + 1))$$
and
$$\delta(2^{K + 1}n) \geq \delta(j) \geq \delta(2^{K + 1}(n + 1)).$$
From almost multiplicativity and the trivial bound,
\begin{align*}
D'_{3p}(\delta(j)) &\leq D'_{3p}(\delta(2^{K + 1}n))D'_{3p}(\delta(j - 2^{K + 1}n))\\
&\lsm_{p, \vep, \dim(C), N(1)} N(2^{K + 1}n)^{\ld(1 - \frac{1}{2^{K + 1}}) + \vep}N(j - 2^{K + 1}n)^{1/2}\\
&\lsm_{p, \vep, \dim(C), N(1)} N(j)^{\ld(1 - \frac{1}{2^{K + 1}}) + \vep}(\frac{N(2^{K + 1}(n + 1))}{N(2^{K + 1}n)})^{1/2}\\
&\lsm_{p, \vep, \dim(C), N(1)} N(j)^{\ld(1 - \frac{1}{2^{K + 1}}) + \vep}N(1)^{2^K}.
\end{align*}

Therefore we have upgraded this estimate to be that for all $j \geq 0$,
$$D'_{3p}(\delta(j)) \lsm_{p, \vep, \dim(C), N(1), \ld}N(j)^{\ld(1 - \frac{1}{2^{K + 1}}) + \vep}.$$
This contradicts the minimality of $\ld$.
\end{proof}

Following the same ideas from the iteration in \cite{li-efficient}, if there is no dependence on $\dim(C)$ and $N(1)$
in \eqref{maininput} (as is the case for our examples in Section \ref{sub:Examples}), the dependence on $\dim(C)$ and $N(1)$ in $D_{3p}(\delta(i))$
is $\exp(\exp(O(\frac{1}{\vep\dim(C)}))\log N(1))$. If there is some dependence on $\dim(C)$ and $N(1)$ in \eqref{maininput}, then an examination
of the proof above shows that this same exact dependence shows up again in $D_{3p}(\delta(i))$.

\section{Decoupling for Cantor subsets of $[0, 1]$}\label{cantor_line}
In Theorem \ref{main}, we reduced the study of decoupling for a Cantor set on the parabola
to that on the line. We now proceed to carefully study the case of $l^{2}L^{2n}$ decoupling
for a Cantor subset of $[0, 1]$.
The use of $2n$ allows us to connect decoupling to number theory.

By rescaling $a$ and $f$, we have that
\begin{align*}
A_{p, m}(S) = \sup\{\nms{\sum_{\ell \in S}a(\ell)e(\ell \cdot x)}_{L^{p}([0, 1]^m)} \mid a: S \rightarrow \R_{\geq 0}, \sum_{\ell \in S}|a(\ell)|^{2} = 1\}
\end{align*}
and
\begin{align*}
K_{p}(\Om) = \sup\{\nms{\sum_{I}f_I}_{L^{p}(\R)} \mid f \textrm{ Schwartz},\sum_{I}\nms{f_I}_{L^{p}(\R)}^{2} = 1\}.
\end{align*}
Making use of that $2n$ is even, we have the following proposition.
\begin{prop}
\label{prop:decoupling-ellipsephic-optimization}
Let $S \subset \Z^m$. Then
\begin{align}
\label{eq:decoupling-ellipsephic-optimization}
A_{2n, m}(S)^{2n}
=
\sup \left\{
\sum_{t \in \Z^m}
\left(
\sum_{\substack{\ell_1, \ldots, \ell_n \in S \\ \ell_1 + \cdots + \ell_n = t}}
\prod_{i=1}^n a(\ell_i)
\right)^2
~\middle|~
a : S \to \R_{\geq 0} \text{ and }
\sum_{\ell \in S} |a(\ell)|^2 = 1
\right\}.
\end{align}
\end{prop}

\begin{proof}
This follows immediately from the observation that
\begin{align*}
\nms{\sum_{\ell \in S}a(\ell)e(\ell \cdot x)}_{L^{2n}([0, 1]^{m})}^{2n} = \left\|\sum_{t \in \Z^m}
\left(
\sum_{\substack{\ell_1, \ldots, \ell_n \in S \\ \ell_1 + \cdots + \ell_n = t}}
\prod_{i=1}^n a(\ell_i)
\right)
e^{2 \pi i t \cdot x}\right\|_{L^{2}([0, 1]^{m})}^{2}
\end{align*}
and then applying Plancherel.
\end{proof}

\subsection{Properties of $A_{2n}(S)$}
For $S \subset \Z^m$ and $S' \subset \Z^{m'}$, we say that $\phi : S \to S'$ is a \emph{Freiman homomorphism of order $n$} if
\begin{align*}
\text{for all } x_1, \ldots, x_n, y_1, \ldots, y_n \in S, \qquad \sum_{i=1}^n x_i = \sum_{i=1}^n y_i \implies \sum_{i=1}^n \phi(x_i) = \sum_{i=1}^n \phi(y_i)
\end{align*}
(see, e.g. \cite[Section 5.3]{tao-vu}). We say that $\phi$ is a \emph{Freiman isomorphism of order $n$} if $\phi$ is a bijection and both
$\phi$ and $\phi^{-1}$ are Freiman homomorphisms of order $n$.

It follows immediately from \Cref{prop:decoupling-ellipsephic-optimization} that if $\phi$ is a bijective Freiman homomorphism of order $n$, then
\begin{equation}
\label{eq:freiman-hom-decoupling}
A_{2n, m}(S) \leq A_{2n, m'}(S'),
\end{equation}
and that \eqref{eq:freiman-hom-decoupling} becomes an equality if $\phi$ is a Freiman isomorphism of order $n$. We also have the following.

\begin{prop}
\label{prop:freiman}
Let $S \subset \Z^m$ and $S' \subset \Z^{m'}$, and let $\phi : S \to S'$ be a bijection. Let
\begin{align}
\label{eq:freiman-D}
D
=
\left\{
\sum_{i=1}^n \phi(x_i) - \sum_{i=1}^n \phi(y_i)
~\middle|~
x_1, \ldots, x_n, y_1, \ldots, y_n \in S
\text{ and }
\sum_{i=1}^n x_i = \sum_{i=1}^n y_i
\right\}
\end{align}
Then
\begin{align}
\label{eq:freiman-dec-ineq}
A_{2n, m}(S) \leq |D|^{\frac{1}{2n}}A_{2n, m'}(S').
\end{align}
\end{prop}
Note that if $\phi$ is a bijective Freiman homomorphism of order $n$, then $D = \{0\}$,
so \eqref{eq:freiman-dec-ineq} becomes \eqref{eq:freiman-hom-decoupling}.
Thus, \Cref{prop:freiman} is a variant of \eqref{eq:freiman-hom-decoupling} for
when the bijection $\phi$ is not a Freiman homomorphism of order $n$, but is ``close'' to being one
(in the sense that $D$ is small).
This proposition should also be compared to \cite[Lemma 2.2]{biggs-quadratic}.

\begin{proof}
Let $a : S \to \R_{\geq 0}$ such that $\sum_{\ell \in S} a(\ell)^2 = 1$. Define $a' : S' \to \R_{\geq 0}$ by $a' = a \circ \phi^{-1}$. Then by the definition of $D$,
\begin{align}
\sum_{t \in \Z^m}
\left(
\sum_{\substack{x_1, \ldots, x_n \in S \\ x_1 + \cdots + x_n = t}}
\prod_{i=1}^n a(x_i)
\right)^2
&=
\sum_{\substack{x_1, \ldots, x_n \in S \\ y_1, \ldots, y_n \in S \\ x_1 + \cdots + x_n = y_1 + \cdots + y_n}}
\left(\prod_{i=1}^n a(x_i)\right) \left(\prod_{i=1}^na(y_i)\right) \nonumber
\\
\label{eq:freiman-expand-sum-S'}
&\leq
\sum_{t \in D}
\sum_{\substack{x_1', \ldots, x_n' \in S' \\ y_1', \ldots, y_n' \in S' \\ \sum_{i = 1}^{n} x_i' - \sum_{i = 1}^{n} y_i' = t}}
\left(\prod_{i=1}^n a'(x_i')\right) \left(\prod_{i=1}^na'(y_i')\right)
\end{align}
Define
\begin{align*}
B(t) =
    \sum_{x_{1}', \ldots, x_{n}' \in S:
    \sum^n_{i=1} x_i' = t
    }
\prod_{i=1}^n a'(x_i')
\end{align*}
so that the right-hand side of \eqref{eq:freiman-expand-sum-S'} is
\begin{align*}
=
\sum_{s, t \in \Z^{m'} : s-t \in D}
B(s)B(t)
&\leq
\sum_{s, t \in \Z^{m'}: s-t \in D}
\frac{B(s)^2+B(t)^2}{2}\\
& = \frac{1}{2}\sum_{\st{s, t \in \Z^{m'}\\s - t \in D}}B(s)^{2} + \frac{1}{2}\sum_{\st{s, t \in \Z^{m'}\\s - t \in D}}B(t)^{2}
\leq
|D|
\sum_{t \in \Z^{m'}}
B(t)^2
\end{align*}
Thus,
\begin{align*}
\sum_{t \in \Z^m}
\left(
\sum_{\substack{x_1, \ldots, x_n \in S \\ x_1 + \cdots + x_n = t}}
\prod_{i=1}^n a(x_i)
\right)^2
\leq
|D|
\sum_{t' \in \Z^{m'}}
\left(
\sum_{\substack{x_1', \ldots, x_n' \in S' \\ x_1' + \cdots + x_n' = t'}}
\prod_{i=1}^n a'(x_i')
\right)^2
\end{align*}
which by \Cref{prop:decoupling-ellipsephic-optimization} implies \eqref{eq:freiman-dec-ineq}.
\end{proof}

\begin{prop}
\label{prop:tensor-decoupling}
For $S \subset \Z^{m}$, $S' \subset \Z^{m'}$,
\begin{align*}
A_{2n, m + m'}(S \times S') = A_{2n, m}(S)A_{2n, m'}(S')
\end{align*}
\end{prop}
\begin{proof}
First, we will show  that
\begin{equation}\label{eq-Xgeq}
    A_{2n, m + m'}(S \times S') \geq A_{2n, m}(S)A_{2n, m'}(S').
\end{equation}
For $a:S\to\R_{\geq0}$ and $a':S'\to\R_{\geq0}$, we define $(a\otimes a'):S\times S'\to \R_{\geq0}$ by $$(a\otimes a')(l,l')=a(l)a'(l').$$ Observe that
\begin{align*}
    \|\sum_{(l,l')\in S\times S'}(a\otimes a')(\ell,\ell')&e((\ell,\ell')\cdot (x,x'))\|_{L^{2n}(\T^{m+m'})}\\
    &=\|\sum_{\ell\in S}a(\ell)e(\ell\cdot x)\|_{L^{2n}(\T^{m})}  \|\sum_{\ell'\in S'}a'(\ell')e(\ell'\cdot x')\|_{L^{2n}(\T^{m'})}
\end{align*}
and
$$\|a\otimes a'\|_{\ell^2(S\times S')}=\|a\|_{\ell^2(S)}\|a'\|_{\ell^2(S')}. $$
We therefore obtain \eqref{eq-Xgeq}.

It now remains to show the reverse inequality
\begin{equation}\label{eq-Xleq}
    A_{2n, m + m'}(S \times S') \leq A_{2n, m}(S)A_{2n, m'}(S').
\end{equation}
Fix $x' \in \T^{m'}$. Then we view $b_{x'}(\ell) := \sum_{\ell' \in S'}a(\ell, \ell')e(\ell' \cdot x')$
as a function of $\ell \in S$. We have
\begin{align*}
\|\sum_{\ell \in S}(\sum_{\ell' \in S'}a(\ell, \ell')e(\ell' \cdot x'))e(\ell \cdot x)\|_{L^{2n}_{x}(\T^m)}^{2n} &= \|\sum_{\ell \in S}b_{x'}(\ell)e(\ell \cdot x)\|_{L^{2n}_{x}(\T^m)}^{2n}\\
& \leq A_{2n, m}(S)^{2n}(\sum_{\ell \in S}|b_{x'}(\ell)|^{2})^{2n/2}.
\end{align*}
Next integrating in $\T^{m'}$ gives
\begin{align*}
\|\sum_{\ell \in S, \ell' \in S'}a(\ell, \ell')&e(\ell' \cdot x')e(\ell \cdot x)\|_{L^{2n}(\T^{m + m'})} \leq A_{2n, m}(S)\|(\sum_{\ell \in S}|\sum_{\ell' \in S'}a(\ell, \ell')e(\ell' \cdot x')|^{2})^{1/2}\|_{L^{2n}_{x'}(\T^{m'})}.
\end{align*}
Since $2n \geq 2$, applying Minkowski's inequality allows us to interchange the $L^{2n}_{x'}$ and the $\ell^2$ sum over $\ell \in S$. Thus the above is controlled by
\begin{align*}
A_{2n, m}(S)(\sum_{\ell \in S}\|&\sum_{\ell' \in S'}a(\ell, \ell')e(\ell' \cdot x')\|_{L^{2n}_{x'}(\T^{m'})}^{2})^{1/2}\leq A_{2n, m}(S)A_{2n, m}(S')(\sum_{\ell \in S, \ell' \in S'}|a(\ell, \ell')|^{2})^{1/2}
\end{align*}
from which \eqref{eq-Xleq} follows.
\end{proof}

\subsection{Arithmetic Cantor sets and ellipsephic sets}\label{arith}
Let
\begin{align}\label{alphadef}
\alpha_{2n}(\Ellipsephic{q}) := \limsup_{j \rightarrow \infty}\frac{\log A_{2n, 1}(\Level[j]{\Ellipsephic{q}})}{\log k^j}
\end{align}
and similarly let
\begin{align}\label{kappadef}
\kappa_{2n}(\ArCantor{q}) := \limsup_{j \rightarrow \infty}\frac{\log K_{2n}(\Level[j]{\mathscr{C}_{q}^{\{d_1, \ldots, d_k\}}})}{\log k^j}.
\end{align}
We call these the decoupling exponents of $A_{2n, 1}(\Level[j]{\Ellipsephic{q}})$ and $K_{2n}(\Level[j]{\mathscr{C}_{q}^{\{d_1, \ldots, d_k\}}})$, respectively.

In this section we will show that from a decoupling point of view the sets $\Level[j]{C_{q}^{\{d_1, \ldots, d_k\}}}$ and $\Level[j]{\Ellipsephic q}$ have similar nature.
Namely, we will prove the following proposition. This allows us to upgrade results obtained from discrete restriction of ellipsephic sets to decoupling for
arithmetic Cantor sets. In particular, later in \Cref{prop:no_carryover} when the ellipsephic set does not have carryover, the discrete
restriction problem has a particularly nice structure.

\begin{prop}
\label{prop:ellipsephic_to_cantor}
For an integer $n \geq 1$,
	\begin{equation}
	\label{eq:ellipsephic_to_cantor}
K_{2n}(\Level[j]{\PartCantor q}) \sim A_{2n, 1}(\Level[j]{\Ellipsephic q})
	\end{equation}
	where the implicit constant is an absolute constant. In particular by \eqref{alphadef} and \eqref{kappadef}, this implies that
$$\kappa_{2n}(\ArCantor{q}) = \alpha_{2n}(\Ellipsephic{q}).$$
\end{prop}

\begin{proof}
Let $E_j := \Level[j]{\Ellipsephic q}$ and $C_j := \Level[j]{C_{q}^{\{d_1, \ldots, d_k\}}}$. For $\ell\in E_j$, %
we will denote by $I_\ell$ the interval $[q^{-j}\ell, q^{-j}(\ell+1)]$, so that $C_j = \bigcup_{\ell \in E_j} I_\ell$.

First we show the $\lesssim$ direction in \eqref{eq:ellipsephic_to_cantor}. Let $f(x)$ be a Schwartz function Fourier supported on $C_j$ such that $\sum_{\ell \in E_j} \| (f *\check 1_{I_\ell})\|_{L^{2 n}(\mathbb R)}^2 = 1$. %
	 Let $f_{\ell} = f*\check 1_{I_{\ell}}$. Note that for $\ell_1, \ldots, \ell_n \in E_j$, the Fourier transform of $\prod_{j = 1}^{n}f_{\ell_i}$ is supported in $[q^{-j}\sum_{i=1}^n\ell_i, q^{-j}(\sum_{i=1}^n\ell_i + n)]$. Therefore, by Plancherel and H\"older,
    \begin{align*}
    \int_{\R}
        |\sum_{\ell\in E_j} f_\ell|^{2n} dx =
    \int_{\R}
        |\sum_{\ell_1, \ldots, \ell_n\in E_j} \prod_{i=1}^n f_{\ell_i} |^{2} dx
        &=
    \int_{\R}
    \sum_{\substack{
    |\sum^n_{i=1} \ell_i - \tilde \ell_i|\le  n \\
    \ell_1, \dots, \ell_n \in E_j \\
    \tilde \ell_1, \dots, \tilde \ell_n \in E_j \\
    }}
    \prod_{i=1}^n f_{\ell_i} \bar f_{\tilde \ell_i} dx
    \\
    &\le
    \sum_{\substack{
    |\sum^n_{i=1} \ell_i - \tilde \ell_i|\le  n \\
    \ell_1, \dots, \ell_n \in E_j \\
    \tilde \ell_1, \dots, \tilde \ell_n \in E_j \\
    }}
    \prod_{i=1}^n \|f_{\ell_i}\|_{L^{2n}(\R)} \| f_{\tilde \ell_i}\|_{L^{2n}(\R)}. \nonumber
    \end{align*}
Then arguing as in the proof of \Cref{prop:freiman}, we have
\begin{align*}
\sum_{t=-n}^n
    \sum_{\substack{
    \sum^n_{i=1} \ell_i - \tilde \ell_i =t \\
    \ell_1, \dots, \ell_n \in E_j \\
    \tilde \ell_1, \dots, \tilde \ell_n \in E_j \\
    }}
    \prod_{i=1}^n \|f_{\ell_i}\|_{2n} \| f_{\tilde \ell_i}\|_{2n}
&\leq
(2n+1)
\sum_{t \in \Z}
\left(
\sum_{\substack{\ell_1, \dots, \ell_n \in E_j \\ \ell_1 + \cdots + \ell_n = t}}
\prod_{i=1}^n \|f_{\ell_i}\|_{2n}
\right)^2
\\
&\leq
(2n+1)  A_{2n, 1}([\Ellipsephic{q}]_{j})^{2n}
\end{align*}
where the last inequality is by \Cref{prop:decoupling-ellipsephic-optimization} and that $\sum_{\ell}\nms{f_{\ell}}_{L^{2n}(\R)}^{2} = 1$.

Next we show the $\gtrsim$ direction in \eqref{eq:ellipsephic_to_cantor}. Let $\phi \in C_{c}^{\infty}(\R)$ be a smooth nonnegative function which is equal
to $cn$ on $[\frac{0.01}{n}, \frac{0.99}{n}]$ and vanishes outside $[0, 1/n]$ and where $c$ is an absolute constant chosen so that
$\nms{\phi}_{1} = 1$. Then observe that $\nms{\phi}_{2} \sim n^{1/2}$ and $\nms{\wc{\phi}}_{\infty} \leq 1$ which imply that $\nms{\wc{\phi}}_{2n} \lsm n^{1/2n}$.

Define $\Phi = \phi^{*n}$, the $n$-fold convolution. Then $\Phi \geq 0$, $\Phi$ is supported in $[0,1]$ and $1 = \| \Phi\|_1 \leq \| \Phi \|_2$. For $\ell \in \Z$, define $\phi_\ell(x) = q^{j} \phi(q^j x - \ell)$. Also define $\Phi_\ell(x) = q^j \Phi(q^j x - \ell)$, so that $\phi_{\ell_1} * \cdots * \phi_{\ell_n} = \Phi_{\ell_1 + \cdots + \ell_n}$
and $\Phi_{\ell}$ is supported on $I_{\ell}$.

Since $E_j$ is finite there is a function $a : E_j \to \R$, which attains the supremum in \eqref{eq:decoupling-ellipsephic-optimization}. Let $a : E_j \to \R$ attain the maximum in \eqref{eq:decoupling-ellipsephic-optimization}. For $\ell \in E_j$, define $f_\ell$ by $\widehat f_\ell = a(\ell) \phi_\ell$. Observe that
\begin{align*}
\sum_{\ell_1, \ldots, \ell_n \in E_j}
\widehat f_{\ell_1} * \cdots * \widehat f_{\ell_n}
=
\sum_{\ell_1, \ldots, \ell_n \in E_j}
\left(\prod_{i=1}^n a(\ell_i)\right) \Phi_{\ell_1+ \cdots + \ell_n}
=
\sum_{t \in \Z}
\left(
\sum_{
    \sum^n_{i=1} \ell_i = t
    }
\prod_{i=1}^n a(\ell_i)
\right)
\Phi_t
\end{align*}
We note that the supports of $\Phi_t$ for $t \in \Z$ are disjoint, and that $\| \Phi_t\|_2^2 \geq q^{j}$, so using Plancherel we obtain
\begin{align}\label{eq:lhs-in-decoupling}
\left\| \sum_{\ell \in E_j} f_\ell \right\|_{2n}^{2n}
&=
\left\|
\sum_{\ell_1, \ldots, \ell_n \in E_j}
\widehat f_{\ell_1} * \cdots * \widehat f_{\ell_n}
\right\|_{2}^{2}
\geq
q^j
\sum_{t \in \Z}
\left(
\sum_{
    \sum^n_{i=1} \ell_i = t
    }
\prod_{i=1}^n a(\ell_i)
\right)^2 =
q^{j}A_{2n, 1}([\Ellipsephic{q}]_{j})^{2n} %
\end{align}
Next, $\|f_\ell\|_{2n} \lsm n^{1/(2n)} |a(\ell)| q^{j/(2n)}$, so
\begin{align}
\label{eq:rhs-in-decoupling}
(\sum_{\ell \in E_j} \| f_\ell \|_{2n}^2)^n
\lsm
nq^{j} (\sum_{\ell \in E_j} |a(\ell)|^2)^{n}
=
nq^{j}
\end{align}
By comparing \eqref{eq:lhs-in-decoupling} with \eqref{eq:rhs-in-decoupling}, we see that
\begin{align*}
A_{2n, 1}([\Ellipsephic{q}]_{j}) \lsm n^{1/(2n)}K_{2n}([\PartCantor{q}]_{j})
\end{align*}
as desired.
\end{proof}

Recall that given an $n$ we say that $[\Ellipsephic{q}]_{j}$ has no carryover
if $nd_k < q$.
In the no carryover case, $A_{2n, 1}(\Level[j]{\Ellipsephic{q}})$ has a particularly nice structure and we are able to characterize
the extremizer of the associated discrete restriction estimate which will allow us the compute the decoupling constant $K_{2n}(\Level[j]{\PartCantor{q}})$.

\begin{prop}\label{prop:no_carryover}
Fix $n \geq 1$. Let $\Ellipsephic{q}$ be an ellipsephic set without carryover.
Let $\operatorname{Digits}_q : \Level[j]{\Ellipsephic{q}}\to \{0, \dots, q-1\}^j$ be the base $q$ expansion of a number. Then
	\begin{equation*}
		A_{2n, 1}( [\Ellipsephic{q}]_j ) = A_{2n, 1}( [\Ellipsephic{q}]_1 )^j,
	\end{equation*}
	and there exists a function $f:\{0,\dots, q-1\} \to \mathbb R_{\geq 0}$ (depending on $q$ and $\digits$) such that, for all $j\in \mathbb N$ the function
	\begin{equation}\label{freimax}
		f_j(x) = \prod_{i=1}^{j} f((\operatorname{Digits}_q(x))_i)
	\end{equation}
	witnesses the value of $A_{2n, 1}( [\Ellipsephic{q}]_j )$ where here we use the notation is that given a vector $(x_{1}, \ldots, x_{j})$, $(x_{1}, \ldots, x_{j})_{i} = x_{i}$ for $1 \leq i \leq j$.
\end{prop}

\begin{proof}
Since there is no carryover, the map $\operatorname{Digits}_q : [\Ellipsephic{q}]_j \to \{d_1, \dots, d_k\}^j$ defined by $\sum_{s=0}^{j-1} a_s q^s\mapsto (a_0,a_1,\ldots,a_{j-1})$ is a
Freiman isomorphism of order $n$. Hence by \eqref{eq:freiman-hom-decoupling} and \Cref{prop:tensor-decoupling},
\begin{align*}
A_{2n, 1}([\Ellipsephic{q}]_j)
=
A_{2n, j}(\{d_1, \dots, d_k\}^j)
=
A_{2n, 1}(\{d_1, \dots, d_k\})^j
.
\end{align*}
Let $f$ be the function which witnesses the value of
\begin{align*}
\sup_{\st{a: \{d_1, \ldots, d_k\} \rightarrow \R_{\geq 0}\\\sum_{\ell \in \{d_1, \ldots, d_k\}}a(\ell)^{2} = 1}}\sum_{t \in \Z}(\sum_{\st{\ell_1, \ldots, \ell_{n} \in \{d_1, \ldots, d_k\}\\\ell_1 + \cdots + \ell_n = t}}\prod_{i= 1}^{n}a(\ell_i))^{2}.
\end{align*}
Such a function exists since $\{d_1, \ldots, d_k\}$ is a finite set.
Finally since $\operatorname{Digits}_{q}$ is a Freiman isomorphism of order $n$, following a proof similar to that of \Cref{prop:freiman}
shows that $f_{j}$ as defined in \eqref{freimax} witnesses the value of $A_{2n, 1}([\Ellipsephic{q}]_{j})$.
\end{proof}

\begin{figure}[!htbp]
	\includegraphics[width=\textwidth]{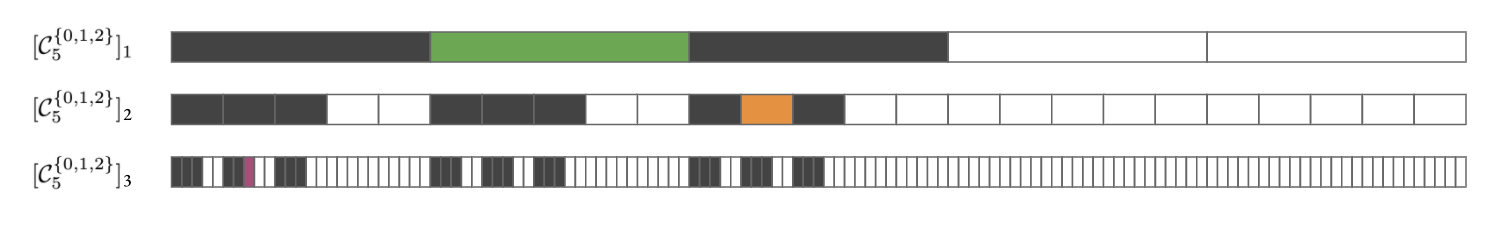}
	\begin{minipage}
		{.6\textwidth}
		\includegraphics[width=\textwidth]{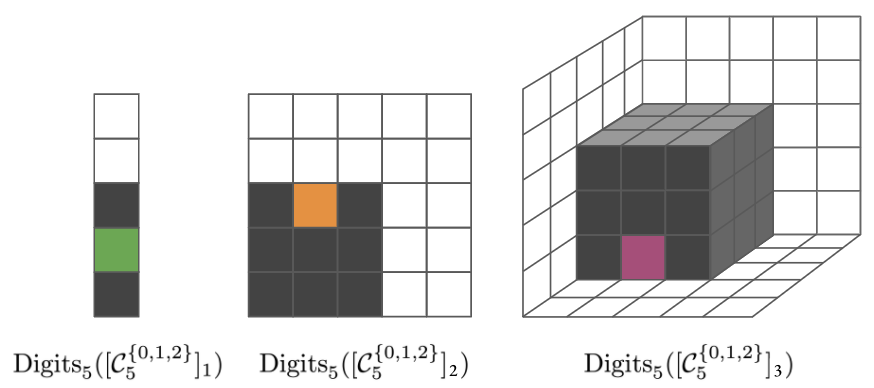}
	\end{minipage}
	\begin{minipage}
		{.39\textwidth}
		\caption{Tensor procedure described in Proposition \ref{prop:no_carryover}. Each digit in the $q$-ary expansion of $\left[\Ellipsephic q\right]_t$ is mapped to its own axis in $\mathbb Z^t$. An element of each $\left[\Ellipsephic[\{0,1,2\}] 5\right]_t$  in the figure has been highlighted both in the digit expansion and the original ellipsephic/Cantor set.}
	\end{minipage}
\end{figure}

As an immediate application of having no carryover,
we now use \Cref{prop:ellipsephic_to_cantor} and \Cref{prop:no_carryover}
to show that the decoupling constant for a Cantor subset in $[0, 1]$
not only depends on the Hausdorff dimension but also arithmetic properties
of the Cantor set.

More precisely we show the following.
\begin{prop}\label{maxexp}
Fix an integer $n \geq 1$ and fix a Hausdorff dimension $d := \frac{\log r}{\log s}$ with $0 < d < 1$ and $r, s \in \N$.
Then there exists an arithmetic Cantor set $C_{q}^{\{d_1, \ldots, d_k\}}$ of dimension $d$ such that
$$\kappa_{2n}(C_{q}^{\{d_1, \ldots, d_k\}}) \geq \frac{1}{2} - \frac{1}{2n}.$$
\end{prop}
\begin{proof}
Let $T$ be large chosen later. Let $D_T := \{1, \ldots, r^T\}$ and $q_T := s^T$. Then $C_{q_T}^{D_T}$ has Hausdorff dimension
equal to $\frac{\log r^T}{\log s^T} = \frac{\log r}{\log s}$.
We can also choose $T$ so large so that $nr^{T} < s^{T}$ and so the associated ellipsephic set $\mathcal{E}_{q_T}^{D_T}$ has no carryover.
Then
\begin{align*}
\kappa_{2n}(C_{q_T}^{D_T}) = \alpha_{2n}(\mathcal{E}_{q_T}^{D_T}) &= \limsup_{J \rightarrow \infty}\frac{\log A_{2n, 1}([\E_{q_T}^{D_T}]_{J})}{\log (r^{T})^{J}}\\
& = \limsup_{J \rightarrow \infty} \frac{\log A_{2n, 1}([\E_{q_T}^{D_T}]_{1})^{J}}{\log (r^{T})^{J}}= \frac{\log A_{2n, 1}([\E_{q_T}^{D_T}]_{1})}{\log r^T}
\end{align*}
where the first equality is an application of \Cref{prop:ellipsephic_to_cantor}, the second equality
is by \eqref{alphadef}, and the third equality is because of \Cref{prop:no_carryover}.
Since if we choose $a(\ell) = 1$, $A_{2n, 1}(\{1, \ldots, r^T\}) \geq (r^{T})^{\frac{1}{2} - \frac{1}{2n}}$, the claim now follows.
\end{proof}
Note that $\kappa_{2n}(C_{q}^{\{d_1, \ldots, d_k\}}) \leq \frac{1}{2} - \frac{1}{2n}$. To see this, one
can either interpolate the estimates $D_{2}(\delta(i)) = 1$ and $D_{\infty}(\delta(i)) \leq N(i)^{1/2}$ (see \cite[Exercise 10$(iv)$]{tao-247} for an interpolation
theorem) or alternatively one can follow the same proof as in \cite[Proposition 1.12]{lguth118} for a direct proof.
Thus \Cref{maxexp} says that even though our Cantor set has small Hausdorff dimension, it can still have a decoupling constant
that is as large as possible.

We had particularly good structure when $\Ellipsephic{q}$ did not have carryover,
however the case when one has carryover is much harder.
In the general case, from a computational standpoint, the following lemma tells us that we can obtain a good approximation
on $\alpha_{2n}(\Ellipsephic{q})$ by estimating $A_{2n, 1}$ on the finite sets $[\Ellipsephic{q}]_{t}$.

\begin{prop}\label{prop: DecExp estimate}
Let $\Ellipsephic{q}$ be an ellipsephic set potentially with carryover. Let $t > \log_{q}n$. Then $\alpha_{2n}(\Ellipsephic{q})$
can be approximated by computing $A_{2n, 1}([\Ellipsephic{q}]_{t})$ with the following bound:
	\begin{equation}
		\label{eq:carryover-inequality}
|\alpha_{2n}(\Ellipsephic{q}) - \frac{\log A_{2n, 1}([\Ellipsephic{q}]_{t})}{\log k^{t}}| \leq \frac{\log(2n + 1)}{2nt\log k}.
	\end{equation}
and therefore
	\begin{equation*}
	\label{eq:carryover-limit}
\alpha_{2n}(\Ellipsephic{q}) = \lim_{t \rightarrow \infty}\frac{\log A_{2n, 1}([\Ellipsephic{q}]_{t})}{\log k^{t}}.
	\end{equation*}
\end{prop}
\begin{proof}
Choose $t \in \N$ such that $q^{t} > n$ and note that
	$$\left [\Ellipsephic[{\left[\Ellipsephic{q}\right]_t}]{q^t}\right ]_j = \left[\Ellipsephic q \right]_{jt}.$$
Consider the bijection
\begin{align}
\label{eq:digit-qt}
\operatorname{Digit}_{q^t}: \left[\Ellipsephic q \right]_{jt} &\longrightarrow  \left[\Ellipsephic q\right]_t^j, \\
\notag  \sum_{s=0}^{j-1} a_s q^{st} &\longmapsto \left(a_0, a_1, \ldots, a_{j-1}\right)
\end{align}
For this map, the set $D$ in \eqref{eq:freiman-D} satisfies
\begin{align}\label{dinclusion}
D
\subset
\{(q^ta_1, q^ta_2-a_1, \ldots, q^ta_{j-1}-a_{j-2}, -a_{j-1}) : a_1, \dots ,a_{j-1} \in \{-n + 1, \dots, n - 1\} \}.
\end{align}
To see this, note that the inverse of $\operatorname{Digit}_{q^t}$ extends to a group homomorphism $\Z^j \to \Z$, so $D$ is contained in the kernel of this group homomorphism. Furthermore, the set $\left[\Ellipsephic q\right]_t$ is bounded above by $q^t-1$. These two observations together imply
\[
D \subset 
\{(b_{0}, \ldots, b_{j-1}) \in \Z^{j} : \sum_{s=0}^{j-1} q^{st} b_s = 0\}
\cap
[-n(q^t-1), n(q^t-1)]^j
.
\]
To show \eqref{dinclusion}, suppose $(b_{0}, \ldots, b_{j-1}) \in D$. Then $|b_s| \leq n (q^t-1)$ and
\begin{align}
\label{eq:sum-bs}
\sum_{s=0}^{j-1} q^{st} b_s = 0.
\end{align}
Taking \eqref{eq:sum-bs} modulo $q^t$ gives $b_0 \equiv 0 \pmod{q^t}$, hence, $b_0 = q^t a_1$ for some $a_1 \in \Z$. Also $|b_0| \leq n(q^t-1)$ implies $|a_1| \leq n-1$. Then taking \eqref{eq:sum-bs} modulo $q^{2t}$ gives $q^ta_1 + q^tb_1 \equiv 0 \pmod{q^{2t}}$, so $b_1 = -a_1 +q^t a_2$ for some $|a_2| \leq n-1$. %
By repeating this, we get $b_{s} = -a_s + q^t a_{s+1}$ for $s = 1, \ldots, j-2$. Finally, \eqref{eq:sum-bs} gives us $b_{j-1} = -a_{j-1}$. (We can think of the numbers $(a_1, \ldots, a_{j-1})$ as ``carryover digits.'')

Equation \eqref{dinclusion} implies $|D| \leq (2n+1)^{j}$. By \Cref{prop:freiman} and \Cref{prop:tensor-decoupling}, this tells us that
\[A_{2n, 1}([\Ellipsephic{q}]_{jt}) \leq (2n + 1)^{\frac{j}{2n}}A_{2n, 1}([\Ellipsephic{q}]_{t})^{j}.\]
Also, note that the inverse of the map \eqref{eq:digit-qt} is a Freiman homomorphism of order $n$, so by \eqref{eq:freiman-hom-decoupling}
\begin{align*}
A_{2n, 1}([\Ellipsephic{q}]_{t})^{j} \leq A_{2n, 1}([\Ellipsephic{q}]_{jt}).
\end{align*}
Applying \eqref{alphadef} to the above two inequalities then proves \eqref{eq:carryover-inequality}.

\end{proof}
\begin{rem}
Note that the right hand side of \eqref{eq:carryover-inequality} is nondecreasing in $t$ (when $n$ and $k$ are kept constant),
so increasing $t$ gives strictly better and better approximations to $\alpha_{2n}(\E_{q}^{\{d_1, \ldots, d_k\}})$.
\end{rem}

\subsection{Examples}
\label{sub:Examples}
The above results in this section allow for explicit computations (in relatively simple cases) and numerical approximations
(in the remaining, more complex cases) of the $l^2 L^{2n}$ decoupling constant associated to an arithmetic Cantor set.

To demonstrate some examples, we consider the $l^{2}L^{4}$ decoupling constant for the following arithmetic Cantor sets.
To study $K_{4}$, we first use \Cref{prop:ellipsephic_to_cantor} to reduce
to studying $A_{4, 1}$. Then we assume $q$ is sufficiently large
so that we are in the no carryover case which allows us to use \Cref{prop:no_carryover} and
\Cref{prop:decoupling-ellipsephic-optimization} which reduces to an optimization problem.

Note that if we take $a(\ell) = 1$ in the definition of $A_{4, 1}$, this amounts to studying the additive energy. In the case of an ellipsephic set, one can apply for example, \cite[Lemma 3.10]{DyatlovJin}. However this would only give a lower bound on $A_{4, 1}$ and the function defined by $a(\ell) = c$ for some $c$ is not always the optimizer of the discrete restriction problem for ellipsephic sets (see for example, \Cref{012q} below).

\begin{ex}[The $(0, 1) \pmod{q}$ arithmetic Cantor set]\label{01q}
Let $k = 2$ and $\{d_1, d_2\} = \{0, 1\}$. At each level $j$, this Cantor set has $2^j$ many intervals. By \Cref{prop:decoupling-ellipsephic-optimization},
\begin{align*}
        A_{4, 1}([\E_{q}^{\{0, 1\}}]_1)^{4}
=
\sup \left\{
(a_0^2)^2
+ (a_0a_1 + a_1a_0)^2
+ (a_1^2)^2
~\middle|~
a_0^2 + a_1^2 = 1
\right\}
=
\frac{3}{2}
\end{align*}
It is easy to see that the maximum is attained when  $a_0 = a_1 = 2^{-1/2}$.
If $q > 2$, then there is no carryover, so \Cref{prop:no_carryover} implies that
\begin{align*}
        K_{4}([\mathscr{C}_{q}^{\{0, 1\}}]_{j})^{4} \sim A_{4, 1}([\E_{q}^{\{0, 1\}}]_{j})^{4} = (3/2)^{j} = (2^{j})^{\log_{2}(3/2)}.
\end{align*}
This should be compared to the trivial bound that $K_{4}([\mathscr{C}_{q}^{\{0, 1\}}]_{j})^{4} \leq 2^{j}$.
\end{ex}

\begin{ex}[The $(0, 2) \pmod{3}$ arithmetic Cantor set]\label{02q}
Let $k = 2$ and $\{d_1, d_2\} = \{0, 2\}$. Then $[C_{q}^{\{0, 2\}}]_{j}$ is the $j$th level of the middle thirds Cantor set.
Since we are studying the $l^{2}L^{4}$ decoupling constant $K_{2 \cdot 2}([\mathscr{C}_{q}^{\{0, 2\}}]_{j})$, $n = 2$ and so the associated ellipsephic set $[\E_{3}^{\{0, 2\}}]_{j}$ has carryover.
However, note for all levels $j$, the map $\phi: [\E_{3}^{\{0, 2\}}]_{j} \rightarrow [\E_{3}^{\{0, 1\}}]_{j}$
given by $x \mapsto x/2$ is a Freiman isomorphism of order $2$ and the latter set does not have carryover.
Therefore from \Cref{prop:ellipsephic_to_cantor},
\begin{align*}
K_{4}([\mathscr{C}_{3}^{\{0, 2\}}]_{j})^{4} \sim A_{4, 1}([\mathcal{E}_{3}^{\{0, 2\}}]_{j})^{4} = A_{4, 1}([\mathcal{E}_{3}^{\{0, 1\}}]_{1})^{4} = (3/2)^{j}
\end{align*}
where the first equality is because of \eqref{eq:freiman-hom-decoupling} and
the second equality is because of Example \ref{01q}.
Therefore we have computed precisely the $l^2 L^{4}$ decoupling constant
for the middle thirds Cantor set.
\end{ex}

\begin{ex}[The $(0, 1, 2) \pmod{q}$ arithmetic Cantor set]\label{012q}
Let $k = 3$ and $\{d_1, d_2, d_3\} = \{0, 1, 2\}$. At each level $j$, this Cantor set has $3^j$ many intervals. By \Cref{prop:decoupling-ellipsephic-optimization},
\begin{align*}
A_{4, 1}&([\E_{q}^{\{0, 1, 2\}}]_{1})^{4}\\
&=
\sup \left\{
(a_0^2)^2
+ (2 a_0a_1)^2
+ (2 a_0a_2 + a_1^2)^2
+ (2 a_1a_2)^2
+ (a_2^2)^2
~\middle|~
a_0^2 + a_1^2 + a_2^2 = 1
\right\}
=
\frac{15}{7}
\end{align*}
One can check that $a_0 = a_2 = (2/7)^{1/2}, a_1 = (3/7)^{1/2}$ attains the maximum.

If $q > 4$, then there is no carryover, so \Cref{prop:no_carryover} implies that
\begin{align*}
K_{4}([\mathscr{C}_{q}^{\{0, 1, 2\}}]_{j})^{4} \sim A_{4, 1}([\E_{q}^{\{0, 1, 2\}}]_{j})^{4} = (15/7)^{j} = (3^{j})^{\log_{3}(15/7)}.
\end{align*}
This once again should be compared to the trivial bound that $K_{4}([\mathscr{C}_{q}^{\{0, 1, 2\}}]_{j})^{4} \leq 3^{j}$.

\end{ex}

\begin{ex}[The $(0,1,3) \pmod{q}$ arithmetic Cantor set]\label{013q}

Let $k = 3$ and $\{d_1, d_2, d_3\} = \{0, 1, 3\}$. At each level $j$, this Cantor set has $3^j$ many intervals. By \Cref{prop:decoupling-ellipsephic-optimization},
\begin{align*}
A_{4, 1}&([\E_{q}^{\{0, 1, 3\}}]_{1})^{4}\\
&=
\sup \left\{
(a_0^2)^2
+ (2 a_0a_1)^2
+ (a_1^2)^2
+ (2 a_0a_3)^2
+ (2 a_1a_3)^2
+ (a_3^2)^2
~\middle|~
a_0^2 + a_1^2 + a_3^2 = 1
\right\}
=
\frac{5}{3}
\end{align*}
One can check that $a_0 = a_1 = a_3 = 3^{-1/2}$ attains the maximum.

If $q > 6$, then there is no carryover, so \Cref{prop:no_carryover} implies that
\begin{align*}
K_{4}([\mathscr{C}_{q}^{\{0, 1, 3\}}]_{j})^{4} \sim A_{4, 1}([\E_{q}^{\{0, 1, 3\}}]_{j})^{4} = (5/3)^{j} = (3^{j})^{\log_{3}(5/3)}.
\end{align*}
As in the previous example, we trivially have that $K_{4}([\mathscr{C}_{q}^{\{0, 1, 3\}}]_{j})^{4} \leq 3^{j}$.
\end{ex}

\begin{ex}[Cantor sets generated by squares]\label{squares}
Let $q > 2$, $S := \{n^2, n \in \mathbb N\}$ the set of squares, and $S_q = S\cap [0,q)$ the squares less than $q$. Then:
\begin{equation}
\label{eq:1.4_in_1d}
    \lim_{q\to \infty} \alpha_{4}(\Ellipsephic[S_q]q) = 0
\end{equation}
By \Cref{main} and the definition of $\alpha$ in \eqref{alphadef}, this implies  \cite[Corollary 1.4]{biggs-quadratic} (note that in \cite{biggs-quadratic},
$q$ is restricted to be a prime number, while here, this restriction is not needed).

Equation \eqref{eq:1.4_in_1d} will follow from \Cref{prop: DecExp estimate} and a number-theoretic estimate about sums of elements in $S$.
Using \eqref{eq:carryover-inequality} with $t=1$ (we can do so since $q > 2$) and using that $\#[\Ellipsephic[S_q]q]_{1} = \lfloor \sqrt{q}\rfloor + 1$, one obtains
\begin{equation*}
|\alpha_{4}(\Ellipsephic[S_q]q) - \frac{\log A_{4, 1}([\Ellipsephic[S_q]q]_1)}{\log (\lfloor \sqrt{q}\rfloor + 1)}| \lsm \frac{1}{\log q}
\end{equation*}
where the implied constant is absolute.
Thus \eqref{eq:1.4_in_1d} will follow from
\begin{equation*}
    \lim_{q\to\infty} \frac{\log A_{4, 1}([\Ellipsephic[S_q]q]_{1})}{\log \sqrt{q}} =0
\end{equation*}
Since counting diagonal solutions shows that $A_{4, 1} \gtrsim 1$, it suffices to show that
\begin{equation}\label{a41q01}
    A_{4, 1}([\Ellipsephic[S_q]q]_{1}) \lesssim q^{o(1)}.
\end{equation}
We in fact show that the left hand side above is $\lsm \exp(O(\frac{\log q}{\log\log q}))$
where the implied constant is absolute.
Indeed, the divisor bound for $\Z[i]$ implies that
\[
\max_{0\le j\le 2q} |\{n_1,n_2\in S, n_1+n_2 = j\}|\leq \exp(O(\frac{\log q}{\log \log q}))
\]
which leads to
\begin{align*}
\sum_{t \in \Z}|\sum_{\st{\ell_1, \ell_2 \in S_{q}:\ell_1 + \ell_2 = t}}a(\ell_1)a(\ell_2)|^{2} &\lsm \exp(O(\frac{\log q}{\log\log q}))\sum_{t \in \Z}\sum_{\st{\ell_1, \ell_2 \in S_{q}:\ell_1 + \ell_2 = t}}|a(\ell_{1})|^{2}|a(\ell_2)|^{2}\\
&= \exp(O(\frac{\log q}{\log\log q}))(\sum_{\ell \in S_{q}}|a(\ell)|^{2})^{2}
\end{align*}
which proves \eqref{a41q01}.
In fact the above proof gives quantitative control on the decoupling exponent and shows $$|\alpha_{4}(\E_{q}^{S_q})| \lsm \frac{1}{\log\log q}$$
where the implied constant is absolute.
\end{ex}

\subsection{Computational results}\label{computational}

\begin{figure}[!t]
\begin{centering}
    \includegraphics[width=.6\textwidth]{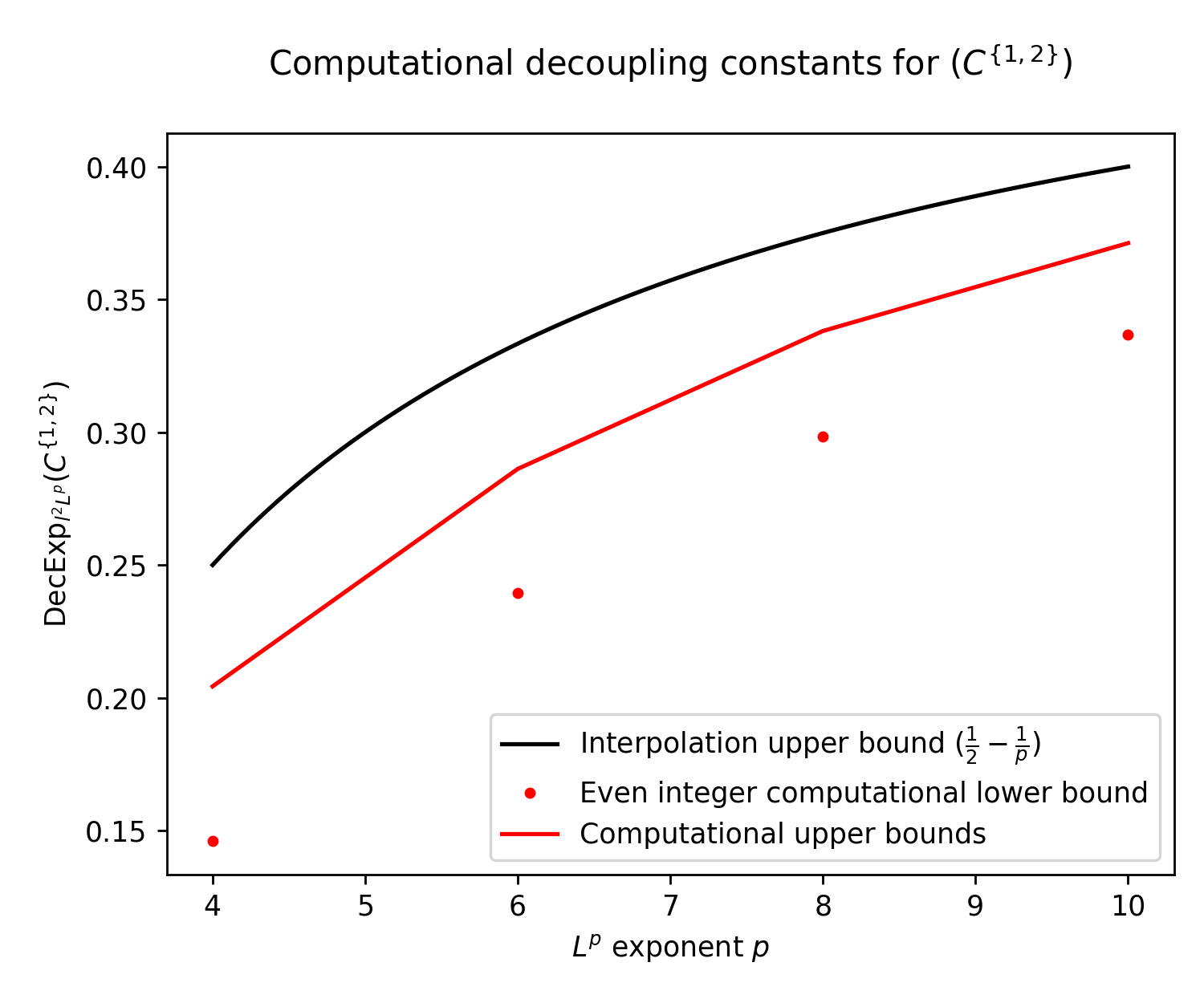}
    \caption{Numerical estimation of $\alpha_{2n}(\E_{3}^{\{1, 2\}})$. The optimization has been performed using gradient descent using Torch. At stopping time the $l^2$ gradients of the optimization where $\le 10^{-8}$. There is no guarantee, however, that the near-local-optimizers are in fact global optimizers of the problem at hand. The upper bounds on the figure (red line) are the upper bounds from \Cref{prop: DecExp estimate} assuming the optimization problem resulted in a global optimizer.}
    \label{fig:computational}
\end{centering}

\end{figure}

Proposition \ref{prop: DecExp estimate} hints of a way of estimating the decoupling exponents of Cantor sets (or at least obtaining an upper bound) by computing the value of $ \frac{\log A_{2n, 1} ([\Ellipsephic q]_t)} {\log k^t}$ for finite values of $t$. Since $[\Ellipsephic q]_t$ contains finitely many points, one may attempt to numerically find the extremizers to the decoupling inequality, in other words, to compute:

\begin{align}
A_{2n, 1} ([\Ellipsephic q]_t)^{2n}
&=
\argmax_{\substack{f \in l^2([\Ellipsephic q]_t)\\
                   \|f\|_{l^2}=1}}
    \sum_{\substack{a_1,\dots a_n \in [\Ellipsephic q]_t \\
                    b_1,\dots b_n \in [\Ellipsephic q]_t \\
                    a_1+\dots+a_n = b_1+\dots +b_n}}
        f(a_1)\dots f(a_n) \cdot \bar f(b_1)\dots \bar f(b_n)\nonumber
\\&= \label{eq:constrained_optim}
\argmax_{\substack{f \in l^2([\Ellipsephic q]_t)\\
                   \|f\|_{l^2}=1}}
    \|\underbrace{f \ast f \dots \ast f}_{n\text{ times}}\|_{l^2(\mathbb Z)}^2
\end{align}

or, as an unconstrained optimization problem,

\begin{equation}
\label{eq:unconstrained_optim}
    A_{2n, 1} ([\Ellipsephic q]_t) =
\argmax_{\substack{\text{supp }f \subseteq [\Ellipsephic q]_t}}\frac{
    \|{f \ast f \dots \ast f}\|_{l^2(\mathbb Z)}^{1/n} }{\|f\|_{l^2(\mathbb Z)}^2}
\end{equation}

We performed the numerical optimization problem in \eqref{eq:unconstrained_optim} for the $(0,2)\mod 3$ Cantor set and $n=1,2,3,4$ using gradient descent. The results can be seen in \Cref{fig:computational}.  While there are no a priori guarantees that the near-local-optimizers obtained from gradient descent are in fact global optimizers of the problem at hand, this method was tested on the previous examples in \Cref{sub:Examples}, and converged to the known decoupling exponent. \\

\subsubsection{A conjectured fixed point method} Studying equation \eqref{eq:constrained_optim}, using Lagrange multipliers one may extract information about the solution, more precisely that, at extremizers (which must exist because $l^2([\Ellipsephic q]_t)$ is a finite-dimensional space) the following equality holds:

\[
f =
\lambda\cdot \chi_{[\Ellipsephic q]_t} \cdot
\nabla
    \|\underbrace{f \ast f \dots \ast f}_{n\text{ times}}\|_{l^2(\mathbb Z)}^2
\]
where $\nabla$ denotes the gradient with respect to $f$ in $l^2([\Ellipsephic q]_t))$.
Let \[\Phi(f):= \lambda\cdot \chi_{[\Ellipsephic q]_t} \cdot
\frac{\partial}{\partial f}
    \|\underbrace{f \ast f \dots \ast f}_{n\text{ times}}\|_{l^2(\mathbb Z)}^2.\]
The functional $\Phi$ sends nonnegative functions to nonnegative functions, and by Cauchy-Schwarz we know there exists an extremizer with nonnegative components. This suggests the following numerical method to compute an extremizer:

\begin{minipage}{\textwidth}
\begin{algorithmic}
\Require TOL $>0$
\Require $f: \chi_{[\Ellipsephic q]_t} \to \mathbb R^+$
\State $n \gets 0$
  \Do
    \State   $f_{n+1} \gets \frac{\Phi(f_n)}{\|\Phi(f_n)\|_2}$
    \State $n\gets n + 1$
  \doWhile{$\|f_{n}-f_{n-1}\|>$TOL}
\end{algorithmic}
\end{minipage}

Convergence of this algorithm to an unique maximum would follow if $f\mapsto \frac{\Phi(f)}{\|\Phi(f)\|}$ was contractive in some norm. Numerical experiments seem to indicate convergence of the algorithm in all situations that were tested at a much faster rate than the gradient descent methods.

\subsubsection{Code} A commented version of the code can be found at \url{https://github.com/jaumededios/Decoupling_Cantor}.

\bibliographystyle{amsplain}
\bibliography{bibliography}
\end{document}